\documentclass[10pt]{article}

\usepackage{amsmath, amssymb, amsthm}
\usepackage{multicol}

\usepackage{graphicx}
\usepackage{color}

\newcommand{\R}{\mathbb{R}}

\newtheorem{theorem}{Theorem}[section]
\newtheorem{lemma}[theorem]{Lemma}
\newtheorem{proposition}[theorem]{Proposition}
\newtheorem{corollary}[theorem]{Corollary}
\newtheorem{remark}[theorem]{Remark}

\DeclareMathOperator{\sign}{sign}

\numberwithin{equation}{section}

\title{Qualitative properties of solutions to mixed-diffusion bistable equations}
\author{Denis Bonheure, Juraj F\"{o}ldes \& Alberto Salda\~{n}a\footnote{D\'{e}partement de Math\'{e}matique, Universit\'{e} libre de Bruxelles, Campus de la Plaine CP 213, Bd. du Triomphe, 1050 Bruxelles, Belgium, dbonheure@ulb.ac.be, jfoldes@ulb.ac.be, asaldana@ulb.ac.be.} \footnote{The authors are supported by MIS F.4508.14 (FNRS). The first author is also supported by INRIA~-- Team MEPHYSTO  \& ARC AUWB-2012-12/17-ULB1-IAPAS}}
\date{}

\begin{document}

\maketitle

\begin{abstract}
We consider a fourth-order extension of the Allen-Cahn model with mixed-diffusion and Navier boundary conditions. Using variational and bifurcation methods,  we 
prove results on existence, uniqueness, positivity, stability, a priori estimates, and symmetry of solutions.  As an application, 
we construct a nontrivial bounded saddle solution in the plane.
\end{abstract}
{\footnotesize
\begin{center}
\end{center}
\begin{center}
\end{center}
}

\section{Introduction}
We study the following fourth-order equation with Navier boundary conditions
\begin{equation}\label{efk:eq}
\begin{aligned}
 \Delta^2 u - \beta \Delta u &= u-u^3 &&\quad \text{ in }\Omega,\\
 u=\Delta u&=0 &&\quad \text{ on }\partial \Omega,
\end{aligned}
\end{equation}
where $\Omega\subset\mathbb R^N,$ $N\geq 1$ is a bounded domain and $\beta>0$.  Such boundary conditions are relevant in many physical contexts \cite{gazzola:2010}
and they permit to rewrite \eqref{efk:eq} as a second order elliptic system with Dirichlet boundary 
conditions. In our best knowledge \eqref{efk:eq} was analyzed only for $N=1$,  see \cite{peletier:2001} and references therein. In this paper we present results on existence, uniqueness, positivity, stability, a priori estimates, regularity, and symmetries of solutions in higher-dimensional domains when $\beta \geq \sqrt{8}$. The case $\beta < \sqrt{8}$ requires different approaches and techniques
and we only prove partial results in this case.

The problem \eqref{efk:eq} is a stationary version of 
\begin{equation}\label{fourth-original}
\partial_{t}u + \gamma \Delta^2 u-\Delta u=u-u^{3}\quad \text{in } \Omega\times [0,\infty), \qquad \gamma>0,
\end{equation}
which was first proposed in 1988 by Dee and van Saarloos \cite{DS} as a higher-order model for physical, chemical, and biological systems. 
The right-hand side of  \eqref{fourth-original} is of bistable type, meaning
it 
has two constant stable states $u \equiv \pm 1$ separated by a third unstable state $u \equiv 0$, see
\cite{DS}.
The distinctive feature of this model is that 
the structure of equilibria is richer compared to its second order counterpart
\begin{equation}\label{FKPP}
\partial_{t}u -\Delta u=u-u^{3}\quad \text{in } \Omega\times[0,\infty),
\end{equation}
giving rise to more complicated patterns and dynamics. 
The equation \eqref{FKPP}
is related to the Fisher-KPP equation (\emph{Fisher-Kolmogorov-Petrovskii-Piscunov} or sometimes simply called \emph{Fisher-Kolmogorov} equation)\footnote{in the original model, the nonlinearity $u^{3}$ is replaced by $u^{2}$}  proposed by Fisher \cite{Fisher} to model the spreading of an advantageous gene and mathematically analyzed by Kolmogorov, Petrovskii, and Piscunov \cite{KPP}.

\medbreak
The equilibria of \eqref{FKPP} satisfy the well-known \emph{Allen-Cahn} or \emph{real Ginzburg-Landau} equation
\begin{equation}\label{second}
-\Delta u=u-u^{3}\quad \text{in } \Omega
\end{equation}
with the associated energy functional
\begin{equation}
\label{2lapfuneps}
\frac{1}{2}\int  |\nabla u|^2  \, dx +\frac{1}{4} \int (|u|^{2}-1)^{2} \, dx,\qquad u\in H^1(\Omega),
\end{equation}
where $H^n(\Omega)=W^{n,2}(\Omega)$ denotes the usual Sobolev space.  This functional is used to describe the pattern and the separation of the (stable) phases $\pm1$ of a material within the van der Waals-Cahn-Hilliard gradient theory of phase separation \cite{CaHi}. For instance, it has important physical applications in the study of interfaces in both gases and solids, e.g. for binary metallic alloys \cite{Allen-Cahn} or bi-phase separation in fluids \cite{Row}. In these models the function $u$ describes the pointwise state of the material or the fluid. The constant equilibria corresponding to the global minimum points $\pm 1$ of the potential $\frac{1}{4}(|u|^{2}-1)^{2}$ are called the pure phases, whereas other configurations $u$ represent mixed states, and orbits connecting $\pm 1$ describe phase transitions. 

To understand the formation of more complex patterns in layering phenomena\textemdash  observed for instance in concentrated soap solutions or metallic alloys\textemdash some nonlinear models for materials include second order derivatives in the energy functional. The basic model can be seen as an extension of \eqref{2lapfuneps}, namely
\begin{equation*}
\int[(\nabla^2u)^2+g(u)\vert\nabla u\vert^2+W(u)]\, dx,\qquad u\in H^2(\Omega),
\end{equation*}
where $\nabla^{2} u$ denotes the Hessian matrix of $u$. It appears  as a simplification of a nonlocal model
\cite{Kawakatsu:1993} analyzed in dimension one in \cite{bonheure:2003,coleman:1992,Leizarowitz:1989,Mizel1998} and in higher dimensions
in \cite{Chermisi:2011,Fonseca:2000,hilhorst:2002}.
In \cite{hilhorst:2002}, the Hessian $\nabla^{2} u$ is replaced by $\Delta u$ as a simplification of the model and it was 
also proposed as model for phase-field theory of edges in anisotropic crystals in \cite{Wheeler:2006}. Finally, we also mention the study of amphiphilic films \cite{leibler:1987} and the description of the phase separations in ternary mixtures containing oil, water, and amphiphile, see \cite{gs}, where the scalar order parameter $u$ is related to the local difference of concentrations of water and oil.

\medbreak

These models motivate the study of the stationary solutions of \eqref{fourth-original}. After scaling, 
equilibria of  \eqref{fourth-original} in the whole space solve
\begin{equation}\label{fourth}
\Delta^2 u-\beta\Delta u=u-u^{3} \quad \text{in } \mathbb R^N.
\end{equation}
We refer to \eqref{fourth} as the \emph{Extended-Fisher-Kolmogorov} equation (EFK) for $\beta>0$ and as the \emph{Swift-Hohenberg} equation for $\beta<0$. This fourth-order model has been mostly investigated for $N=1$, i.e.,
\begin{equation}\label{eq:sta}
u'''' - \beta u'' + u^3-u=0 \quad \text{in } \mathbb R.
\end{equation}
When $\beta\in [\sqrt{8},\infty)$, there is a full classification of 
bounded solutions of (\ref{eq:sta}), which mirrors that of the second order equation. 
Specifically, each bounded solution is either constant, 
a unique kink (up to translations and reflection), or 
a periodic solution indexed by the first integral, whereas there are no pulses. 

For $\beta\in [0,\sqrt{8})$, infinitely many kinks, pulses, and chaotic solutions appear. 
Some solutions can be characterized by homotopy classes, 
but a full classification is not available. The threshold $\sqrt{8}$ is related to a change in stability of constant states $u=\pm 1$. 
The proof of these results rely on purely one-dimensional techniques, for instance, stretching arguments, 
phase space analysis, shooting methods, first integrals, etc.
For more details on the one-dimensional EFK we refer to \cite{BS,peletier:2001} and the references therein.

For $N\geq 2$ let us mention \cite{BH}, where the authors prove an analog of the Gibbon's conjecture and some Liouville-type results. Other related results can be found in \cite{Chermisi:2011,Fonseca:2000,hilhorst:2002}. We also mention that the differential operator from \eqref{fourth} appears in geometry in the study of the \emph{Paneitz-Branson operator}, see for instance \cite{branson,Bakri}, and as a special case of some elliptic systems, see e.g. \cite{sirakov}. In the context of Schr\"odinger equations in nonlinear optics, the fourth order operator in \eqref{fourth} is used to model a mixed dispersion, see e.g. \cite{FiIlPa,BonNasc}.  

\medbreak

To introduce our main results denote
\begin{align}\label{H:def}
H:=H^2(\Omega)\cap H_0^1(\Omega) 
\end{align}
the Sobolev space associated with Navier boundary conditions (see \cite{gazzola:2010} for a survey on Navier and other boundary conditions) and let $J_\beta:H\to\mathbb R$ be the energy functional given by
\begin{align}\label{J0:definition}
J_\beta(u):=\int_\Omega \left(\frac{|\Delta u|^2}{2} + \beta \frac{|\nabla u|^2}{2} + \frac{u^4}{4} - \frac{u^2}{2}\right) \, dx\qquad \text{ for }u\in H.
\end{align}
Any critical point $u$ of $J_\beta$ is a weak solution of \eqref{efk:eq}, that is, $u$ satisfies
\begin{align*}
\int_\Omega \Delta u\Delta v + \beta \nabla u\nabla v + (u^3-u)v \, dx=0 \qquad \text{ for all }v\in H.
\end{align*}
Moreover, we say that $u$ is \emph{stable} if 
\begin{align*}
J_{\beta}''(u)[v,v] = \int_\Omega |\Delta v|^2+\beta |\nabla v|^2 +(3u^2-1)v^2\, dx\geq 0 \quad \text{ for all }v\in H
\end{align*}
and \emph{strictly stable} if the inequality is strict for any $v \not \equiv 0$.

Extracting qualitative information even for global minimizers of \eqref{J0:definition} is far from trivial, since many important tools used in second order problems are no longer available. For example, one cannot use arguments involving the positive part $u^+:=\max\{u,0\}$, absolute value, or rearrangements of functions since they do not belong to $H^2(\Omega)$ in general. 
Furthermore, the validity of maximum principles (or, more generally, positivity preserving properties) is a delicate issue in fourth-order problems and does not hold in general.

For the rest of the paper, $\lambda_1(\Omega)=\lambda_1>0$ denotes the first Dirichlet eigenvalue of $-\Delta$ in $\Omega$ and a \textit{hyperrectangle} is a product of $N$ bounded nonempty open intervals. 

The following is our main existence and uniqueness result, for the second-order counterpart, we refer to \cite{berestycki:1981}.

\begin{theorem}\label{main:theorem}
Let $\beta>0$ and $\Omega\subset \mathbb \R^N$ with $N\geq 1$ be a smooth bounded domain or a hyperrectangle. If 
$\lambda_1^2 + \beta\lambda_1 \geq 1$, then $u\equiv 0$ is the unique weak solution of \eqref{efk:eq}.  If 
\begin{equation}\label{eq:mas}
\lambda_1^2+\beta\lambda_1 < 1 \,,
\end{equation}
then
\begin{enumerate}
 \item[1)] there is $\varepsilon>0$ such that \eqref{efk:eq} admits a positive classical solution for each $\beta\in(\bar\beta-\varepsilon,\bar\beta)$, where $\bar\beta=\frac{1-\lambda_1^2}{\lambda_1}.$
 \item[2)] for each $\beta\geq \frac{\sqrt{8}}{(\sqrt{27}-2)^{1/2}} $ there is a positive classical solution $u$ of \eqref{efk:eq} such that $\|u\|_{L^\infty(\Omega)} \leq \frac{1}{\beta^2}(\frac{4+\beta^2}{3})^{\frac{3}{2}}$ and $\Delta u < \frac{\beta}{2}u$ in $\Omega$.
  \item[3)] for every $\beta\geq \sqrt{8}$ there exists a unique positive solution $u$ of \eqref{efk:eq}. Moreover, this solution is  strictly stable and satisfies $\|u\|_{L^\infty(\Omega)}\leq 1$.
  \end{enumerate}
\end{theorem}

The smoothness assumptions on $\Omega$ are needed for higher-order elliptic regularity results. We single out hyperrectangles 
to use them in the construction of saddle solutions and patterns.
Indeed, by reflexion,  positive solutions of \eqref{efk:eq} in regular polygons that tile the plane give rise to periodic planar patterns.

The quantities involved in Theorem \ref{main:theorem} 2) are of a technical nature. Observe that \eqref{eq:mas} holds for all big enough domains.  As mentioned above, the threshold $\sqrt{8}$ is related to a change in the stability of constant states $u=\pm 1$. For $\beta \geq \sqrt{8}$ the states are saddle-node type whereas for $\beta < \sqrt{8}$ they are saddle-focus type. Hence in the latter case due to oscillations one can not expect $u$ being bounded by $1$. Such oscillations around one can be proved for radial global minimizers arguing as in \cite[proof of Theorem 6]{Bonheure:2004}. 
Intuitively, for $\beta\geq \sqrt{8}$ the Laplacian is the leading term and the equation inherits the dynamics of the second order Allen-Cahn equation, while for $\beta\in(0,\sqrt{8})$ the bilaplacian increases its influence resulting in a much richer and complex set of solutions.
We present numerical approximations\footnote{Computed with FreeFem++ \cite{hecht} and Mathematica 10.0, Wolfram Research Inc., 2014.} of positive solutions using minimization techniques in Figure \ref{figure1} below.
 
\begin{figure}[h!]
  \centering
\includegraphics[height=5cm]{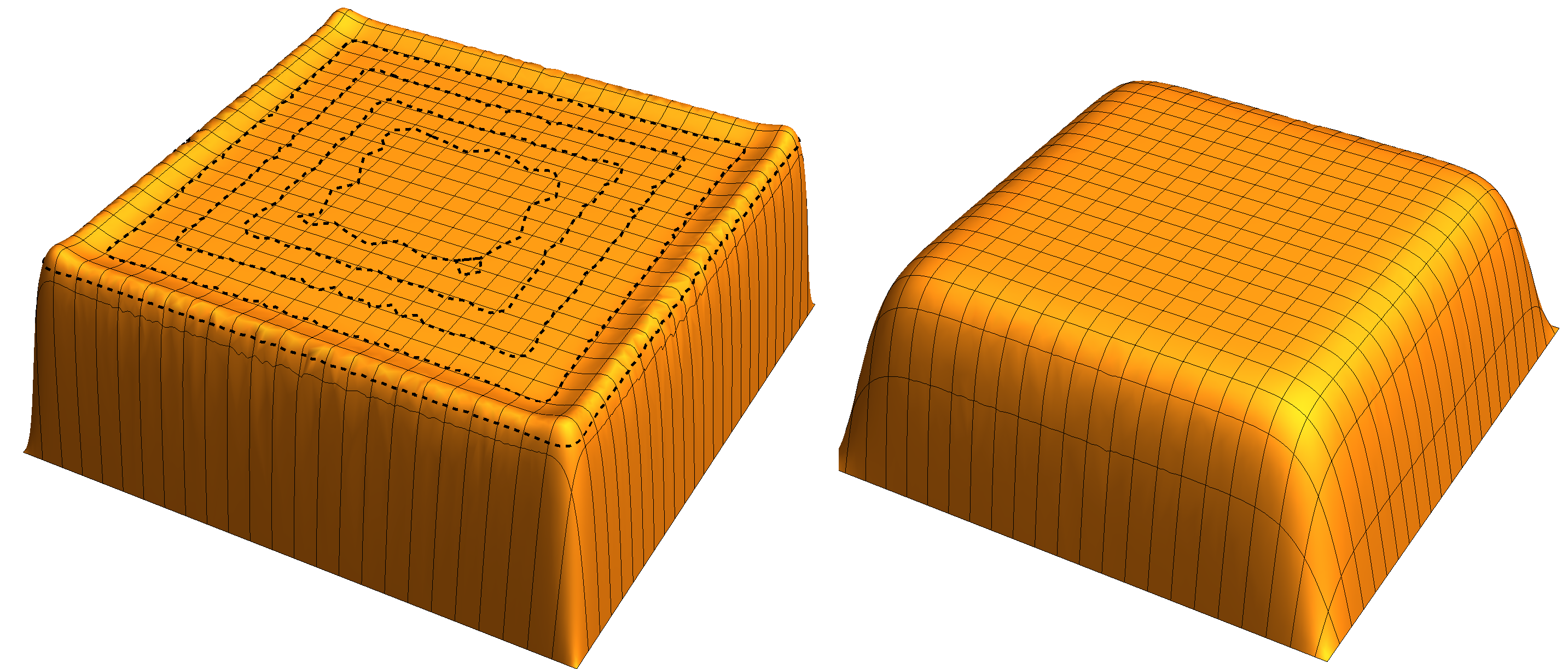}
  \caption{\footnotesize{Numerical approximation of the global minimizer of \eqref{J0:definition} for $\Omega=[0,50]^2$ with $\beta=0.1$ (left) and $\beta=4$ (right). The dotted lines represent the level set $\{u=1\}$.}}
  \label{figure1}
\end{figure}

Note that Theorem \ref{main:theorem} 1) holds for any $\beta>0$, but only for appropriate values of $\lambda_1$. 

Theorem \ref{main:theorem} follows directly from Theorem \ref{existence:positive:solution} and Theorem \ref{bifurcation:beta}. The proof is based on variational and bifurcation techniques. For the variational part (Theorem \ref{existence:positive:solution}), we minimize an auxiliary problem for which we can guarantee the sign and $L^\infty$ bounds of global minimizers. Next, we prove that global minimizers of the auxiliary problem are solutions to \eqref{J0:definition}. The uniqueness is proved using stability, maximum principles, and bifurcation from a simple eigenvalue (Theorem \ref{bifurcation:beta}). After the paper was accepted Tobias Weth suggested us an alternative proof for uniqueness, see Remark \ref{tobias}.

We depict a numerical approximation\footnote{Computed with AUTO-07P \cite{auto}.}
of  the bifurcation branch in Figure \ref{figure2}.
This branch can be  continued even for $\beta < 0$, see Section \ref{numerics} for an example of such a branch and we refer again to \cite{peletier:2001} for a survey on \eqref{eq:sta} for $\beta<0$.  See also Remark \ref{bif:pts} for a discussion on the explicit values of the bifurcation points.  

\begin{figure}[h!]
  \centering
\includegraphics[height=4.5cm]{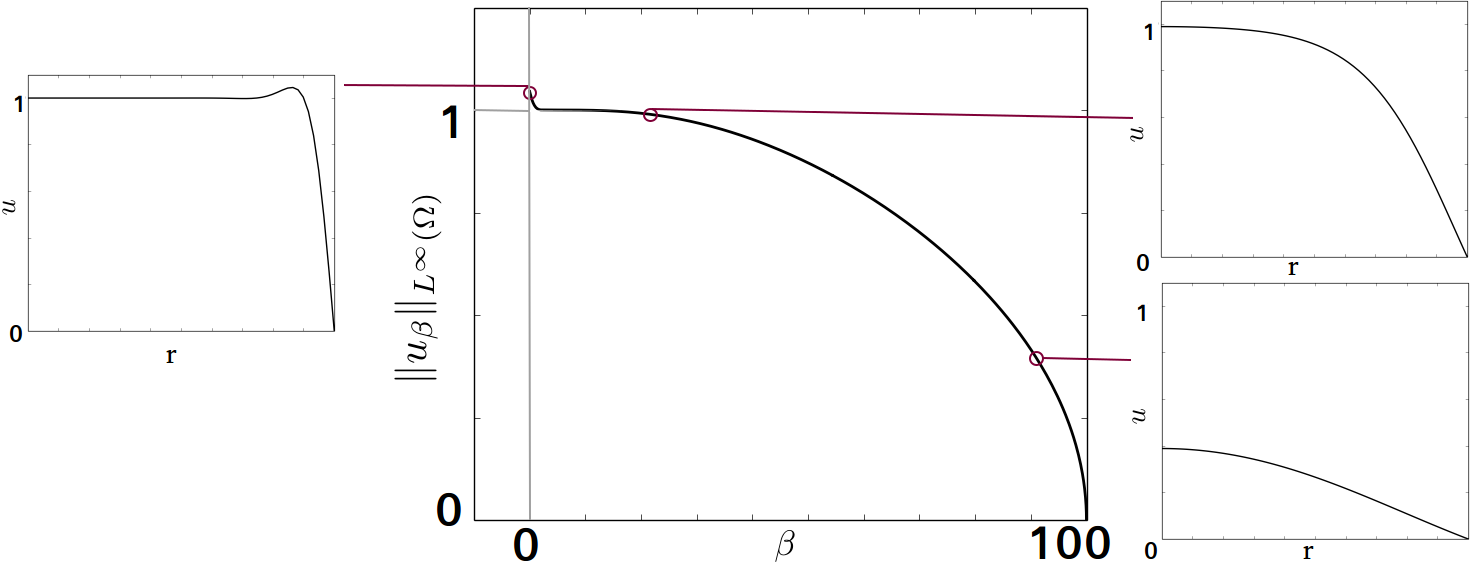}
  \caption{\footnotesize{Numerical approximation of the bifurcation branch and some radial solutions. Here $\Omega$ is a ball in $\mathbb R^2$ of radius $240.483$. 
  }}
  \label{figure2}
\end{figure}

We now use the solution given by Theorem \ref{main:theorem} to construct a saddle solution for \eqref{efk:eq}. We call $u$ a \emph{saddle solution} if $u\not\equiv 0$ and $u(x,y)xy\geq 0$ for all $(x,y)\in\mathbb R^2$. See Figure \ref{saddlefigure} below.

\begin{theorem}\label{saddle:solution:theorem}
For $\beta\geq \sqrt{\frac{8}{\sqrt{27}-2}}$ the problem $\Delta^2 u - \beta \Delta u = u - u^3$ in $\mathbb R^2$ has a saddle solution.
\end{theorem}

\begin{figure}[h!]
  \centering
\includegraphics[height=3.5cm]{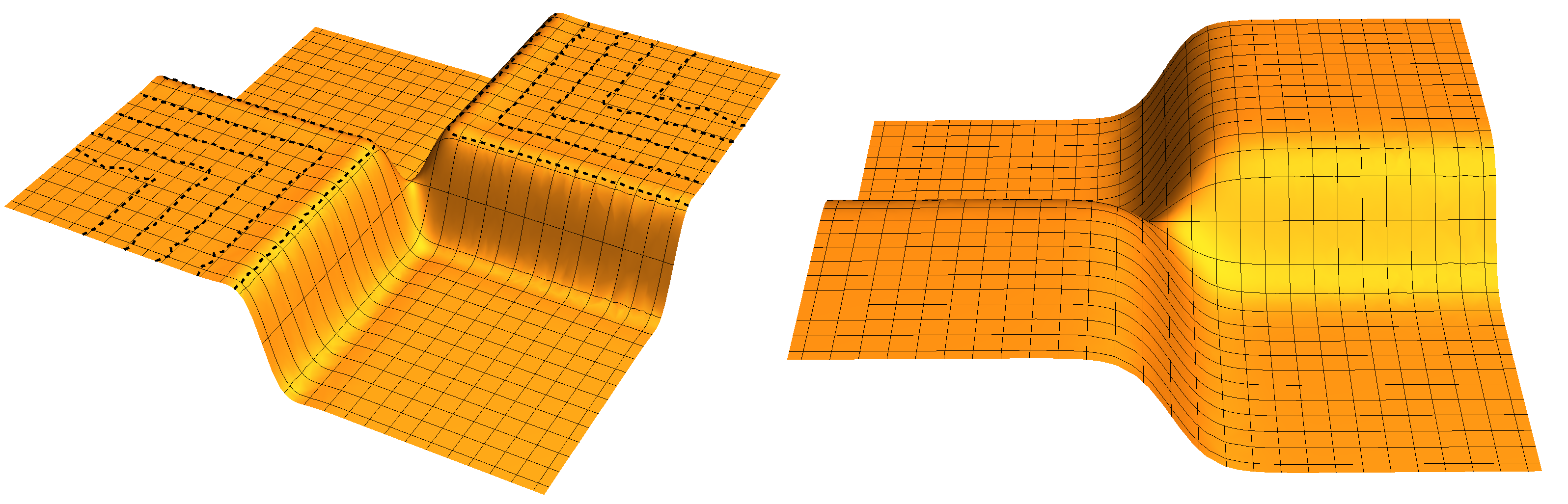}
  \caption{\footnotesize{Saddle solutions for $\beta<\sqrt{8}$ (left) and $\beta\geq\sqrt{8}$ (right). The dotted lines are the level set $\{u=1\}$.}}
  \label{saddlefigure}
\end{figure}

We refer to \cite{cabre:2012,dang:1992} for more information on saddle solutions for second order bistable equations. 

In the following we explore properties of positive classical solutions with a special focus on stability and symmetry properties. 
%
%
%
 
\begin{theorem}\label{theorem:radial:minimizers}
 Let $\Omega$ be a ball or an annulus and let $u$ be a stable solution of \eqref{efk:eq} with $\beta> \sqrt{12}-2\lambda_{1}$ 
 such that $\|u\|_{L^\infty(\Omega)}\leq 1.$  Then $u$ is a radial function. 
\end{theorem}

Note that Theorem \ref{theorem:radial:minimizers} does not assume positivity of solutions. 
We believe that the restriction on $\beta$ is of a technical nature, but it is needed in our approach, see also Remark \ref{beta:res}.

More generally, for reflectionally symmetric domains  we have the following. We say that a domain is \emph{convex and symmetric in the $e_1-$direction} if for every $x=(x_1,x_2,\ldots,x_N)\in \Omega$ we have $\{(tx_1,x_2,\ldots,x_N)\::\: t\in[-1,1]\}\subset\Omega.$

\begin{proposition}\label{symmetry:theorem}
Let $\beta\geq \sqrt{8}$ and let $\Omega\subset \R^N$ be a hyperrectangle or a bounded smooth domain which is convex and symmetric in the $e_1-$direction.  Then, any positive solution of \eqref{efk:eq} satisfies  
 \begin{align*}
u(x_1,x_2,\ldots,x_N)&=u(-x_1,x_2,\ldots,x_N)\qquad \text{ for all }x=(x_1,\ldots,x_N)\in\Omega,\\
\partial_{x_1}u(x)&<0 \qquad \text{ for all }x\in \Omega \text{ such that }x_1>0.
  \end{align*}
\end{proposition}
The proof of Proposition \ref{symmetry:theorem} follows a moving-plane argument for systems, see Figure \ref{figure1} (right)
for the described symmetry.
Note that for $\beta \in (0,\sqrt{8})$ the solution oscillates when close to 1 in big enough domains, in particular it is not monotone, although it may still be symmetric.
In \emph{balls}, Proposition \ref{symmetry:theorem} implies Theorem \ref{theorem:radial:minimizers} with the stability assumption replaced by positivity of the solution.

\medskip

In the following we focus on properties of particular solutions, namely, global minimizers.  The next theorem states positivity of global radial minimizers, that is, functions $u\in H_{r}:=\{v\in H:
v \text{ is radial in }\Omega\}$ such that $J_\beta(u)\leq J_\beta(v)$ for all $v\in H_r$.

\begin{theorem}\label{positivity:minimizers}
Let $\Omega$ be a ball or an annulus, $\beta>0$, \eqref{eq:mas} hold, and let $u\in H_r$ be a global radial minimizer of \eqref{J0:definition}
with
$\|u\|_{L^\infty(\Omega)}\leq 1$. Then $\partial_r u$ does not change sign if $\Omega$ is a ball
and $\partial_r u$ changes sign exactly once if $\Omega$ is an annulus. 
\end{theorem}

Note that global minimizers satisfy the required bound for $\beta \geq \sqrt{8}$, see Proposition \ref{bound:global:minimizers} below. 
The proof of Theorem \ref{positivity:minimizers} relies on a new \emph{flipping technique} that preserves differentiability while diminishing the energy and it is therefore well suited for variational fourth-order problems. Theorem \ref{positivity:minimizers} clearly implies that global radial minimizers do not change sign and we conjecture that this property holds in general, even for $\beta < \sqrt{8}$. For $\beta$ large we can relax the assumptions on the solution, as stated in the following.

\begin{corollary}
Let $\Omega$ be a ball or an annulus, $\beta>\sqrt{12}-2\lambda_{1}$, \eqref{eq:mas} hold,  and let $u$ be a 
global minimizer of \eqref{J0:definition} in $H$. Then $u$ is radial and 
does not change sign in $\Omega$. Moreover, $\partial_r u$ does not change sign if $\Omega$ is a ball while $\partial_r u$ changes sign exactly once if $\Omega$ is an annulus.
\end{corollary}

\begin{proof}
From $\beta>\sqrt{12}-2\lambda_{1}$ and \eqref{eq:mas} follows $\beta \geq \sqrt{8}$ and 
from Proposition \ref{bound:global:minimizers}, we deduce that $\|u\|_{L^\infty(\Omega)}\leq 1$. Theorem \ref{theorem:radial:minimizers} therefore implies that $u$ is radial. Then $u$ is in particular the global minimizer in $H_r$ and the corollary follows from Theorem \ref{positivity:minimizers}.

\end{proof}

As last theorem we present a uniqueness and continuity result for 
\begin{equation}\label{efk:gamma}
\begin{aligned}
 \gamma \Delta^2 u - \Delta u &= u-u^3 &&\quad \text{ in }\Omega,\\
 u=\Delta u&=0 &&\quad \text{ on }\partial \Omega \,,
\end{aligned}
\end{equation}
when $\gamma \to 0$.

\begin{theorem}\label{convergence:theorem} Let $\Omega\subset \R^N$ be a smooth bounded domain
with the first Dirichlet eigenvalue $\lambda_1 < 1$.  Let $\gamma\geq 0$ and $u_\gamma$ be a global minimizer in $H$ of 
\begin{align}\label{J:gamma}
\int_\Omega \left(\gamma \frac{|\Delta u|^2}{2} + \frac{|\nabla u|^2}{2} + \frac{u^4}{4} - \frac{u^2}{2}\right) \, dx.
\end{align}
There is $\delta(\Omega)=\delta>0$ such that, for all $\gamma\in [0,\delta]$, $u_\gamma$ is the unique global minimizer in $H$ and $u_\gamma>0$ in $\Omega$. Moreover, the function $[-\delta,\delta]\to C^{4}(\Omega);$ $\gamma\mapsto u_\gamma$ is continuous and $u_0$ is the global minimizer of \eqref{J:gamma} in $H^1_0(\Omega)$ with $\gamma=0$.
\end{theorem}
The proof of this singular perturbation result relies on regularity estimates which are independent of the parameter $\gamma$, see Lemma \ref{l:reg:2}.

The paper is organized as follows. In Sections \ref{auxiliary:section}, \ref{a:priori:section}, and \ref{bounds:section} we prove crucial auxiliary lemmas for obtaining a priori estimates of solutions.
  Section \ref{positive:solutions:section} contains the proof of the variational part of Theorem \ref{main:theorem}.
Our study of the stability of solutions is contained in Section \ref{section:stability} and \ref{radial:symmetry:section}. In Section \ref{positivity:radial:section} we present our flipping method and prove Theorem \ref{positivity:minimizers}.  The proof of Proposition \ref{symmetry:theorem} can be found in Section \ref{symmetry:section} and the saddle solution is constructed in Section \ref{saddle:section}. The bifurcation result involved in the proof of Theorem \ref{main:theorem} is contained in Section \ref{bifurcation:section} and the continuity result, Theorem \ref{convergence:theorem}, is proved in Section \ref{convergence:section}. Finally, in Section \ref{numerics} we include two numerical approximations of bifurcation branches.

\medskip
To close this Introduction, let us mention that the case $\beta < \sqrt{8}$ requires a different approach due to possible oscillations around 1. In particular, in this case 
strict stability, uniqueness, and symmetry properties of positive solutions and global minimizers are not known.

\vspace{.2cm}

 \noindent \textbf{Acknowledgements:}
We thank Guido Sweers and Pavol Quitter for very helpful discussions and suggestions related to the paper.  We also want to thank Christophe Troestler for introducing us to FreeFem++.

\section{Auxiliary lemma}\label{auxiliary:section}

We prove a very helpful lemma that allows us to obtain a priori bounds on classical solutions. These arguments were used simultaneously in \cite[Lemma 3.1.]{BH} to obtain similar a priori estimates in $\R^N$.

For the rest of the paper we denote by $C_0(\bar{\Omega})$ the space of continuous functions in $\overline{\Omega}$ vanishing on $\partial \Omega$. 

\begin{lemma}\label{bounds:lemma}
 Let $\Omega\subset \mathbb  R^N$ be a bounded domain, $\beta>0$, $f:\mathbb R\to\mathbb R$ a continuous function satisfying $f(0)=0$, and let 
 $u\in C^4(\Omega)\cap C_0(\bar{\Omega})$ be a solution of $\Delta^2u-\beta \Delta u = f(u)$ in $\Omega$ such that $\Delta u\in C_0(\bar{\Omega})$. Set $\overline{u}:=\max\limits_{\overline{\Omega}}u,$ $\underline{u}:=\min\limits_{\overline{\Omega}}u,$ and $g:\mathbb R\to\mathbb R$ given by $g(s):=\frac{4}{\beta^2}f(s)+s$. Then 
\begin{align}\label{bounds:2}
\overline{u}\leq \max\limits_{[\underline{u},\overline{u}]}g\qquad \text{ and }\qquad \underline{u}\geq \min\limits_{[\underline{u},\overline{u}]}g.
\end{align}
Moreover,
\begin{enumerate}
 \item If $\overline{u}\leq \max\limits_{[0,\overline{u}]}g$ and $f\leq 0$ in $(1,\infty),$ then $\overline{u}\leq \max\limits_{[0,1]}g.$
 \item If $\underline{u}\geq \min\limits_{[\underline{u},0]}g$ and $f\geq 0$ in $(-\infty,-1),$ then $\underline{u}\geq \min\limits_{[-1,0]}g.$
 \end{enumerate}
\end{lemma}

\begin{proof}
Let $w\in C^2(\Omega)\cap C_0(\bar{\Omega})$ be given by $w(x):=-\Delta u(x)+\frac{\beta}{2} u(x)$. We prove only the second inequality in \eqref{bounds:2} 
as the first one follows similarly. Fix $x_0,\xi_0\in\overline{\Omega}$ such that $u(x_0)=\underline{u}$ and $w(\xi_0)=\min\limits_{\overline{\Omega}}w.$ If $x_0\in\partial \Omega$, then
\begin{equation*}
\underline{u} = u(x_0) =  0 = g(u(x_0)) \geq \min\limits_{[\underline{u},\overline{u}]}g
\end{equation*}
and the second inequality in \eqref{bounds:2} follows. If $\xi_0\in\partial \Omega$, then $-\Delta u+\frac{\beta}{2} u\geq 0$ in $\Omega$ and the maximum principle implies that $\underline u=0$ and the second inequality in \eqref{bounds:2} follows. If $x_0,\xi_0\in\Omega$, then  
$w(\xi_0)\leq w(x_0)=-\Delta u(x_0)+\frac{\beta}{2} u(x_0)\leq \frac{\beta}{2} u(x_0)$.
Since $-\Delta w(\xi_0) \leq 0$,
\begin{equation*}
\begin{aligned}
 f(u(\xi_0))+\frac{\beta^2}{4}u(\xi_0)&=\Delta^2 u(\xi_0)-\beta\Delta u(\xi_0)+\frac{\beta^2}{4}u(\xi_0)\\
 &=-\Delta w(\xi_0)+\frac{\beta}{2}w(\xi_0)\leq \frac{\beta^2}{4}u(x_0),
\end{aligned}
\end{equation*}
that implies $\underline{u}\geq \min\limits_{[\underline{u},\overline{u}]} g$.

We now prove claim 2. only as claim 1. is analogous. Assume $\underline{u}\leq -1$, otherwise the statement is trivial. Since $f\geq 0$ in $(-\infty, -1)$, then $g(s) \geq s$ in $(- \infty, -1)$, and therefore $\min \limits_{[\underline{u}, -1]}g \geq \min \limits_{[\underline{u}, -1]}s=\underline{u}$. Thus, if $\underline{u}\geq \min\limits_{[\underline{u}, 0]}g$ then $\underline{u}\geq\min\limits_{[-1, 0]}g$ as claimed. 

\end{proof}

\subsection{Regularity}

In this section we prove two regularity results. The first one is a rather standard 
application of known arguments in the fourth order setting and we include details for reader's
convenience. The second lemma is more subtle and the crucial point is the dependencies of the constants involved.

Recall the definition of the space $H$ given in \eqref{H:def}.

\begin{lemma}\label{l:reg}
Let $\Omega\subset \mathbb \R^N,$ $N\geq 1$ be a smooth bounded domain or hyperrectangle 
and fix $\beta, \gamma > 0,$ and $f \in L^\infty(\Omega)$. Let $u \in H$ be a weak solution of $\gamma \Delta^2 u - \beta \Delta u = f$, that is,
 \begin{equation}\label{wkk}
 \int_{\Omega} \gamma \Delta u \Delta \phi + \beta \nabla u \nabla \phi - f \phi \, dx = 0 \qquad \forall  \phi \in H.
 \end{equation} 
Then for each $p > 1$ one has $u \in W^{4, p}(\Omega)\cap C^3(\overline{\Omega})$ with $u = \Delta u = 0$ on $\partial \Omega$ and 
\begin{equation*}
\|u\|_{W^{4, p}(\Omega)} \leq C \|f\|_{L^{\infty}(\Omega)},
\end{equation*} 
where $C = C(\beta, \gamma, p, \Omega)$.  In addition, if $f \in C^{\alpha}(\Omega)$ for some $\alpha \in (0, 1),$ then $u \in C^{4, \alpha} (\overline{\Omega})$ with $\Delta u \in C_0(\overline{\Omega})$ and 
\begin{equation}\label{ret:sch}
\|u\|_{C^{4,\alpha}(\Omega)} \leq \tilde C \|f\|_{C^{\alpha}(\Omega)},
\end{equation} 
where $\tilde C = \tilde C(\beta, \gamma, p, \Omega)$.
\end{lemma}

\begin{proof}
Throughout the proof $C>0$ denotes possibly different constants depending only on $\beta$, $\gamma$, $p$, and $\Omega$.

Assume first that $\Omega\subset \mathbb R^N,$ $N\geq 1,$ is a smooth bounded domain. By the Riesz representation theorem there are weak solutions $\bar{u}, \bar{v}\in H^1_0 (\Omega)$ of the equations
\begin{equation}\label{sys}
\begin{aligned}
- \Delta \bar{v} + \frac{\beta}{\gamma} \bar{v} = f \quad\text{ and }\quad  - \gamma \Delta \bar{u} = \bar{v}\qquad  \textrm{in } \Omega.
\end{aligned}
\end{equation}

Fix $p>1$.  Then, by \cite[Lemma 9.17]{gilbarg:2001} and \cite[Ch 9 Sec 2 Thm 3]{krylov:2008},
$\|\bar u\|_{W^{4, p}(\Omega)} \leq C \|\bar v\|_{W^{2, p}(\Omega)}  \leq C\|f\|_{L^{p}(\Omega)}$, which gives the first estimate in the statement.  

By embedding theorems, we have $\bar{u} \in C^3(\overline{\Omega})$ and $\bar{u} = 0 = \bar{v} = -\gamma \Delta \bar{u}$ on $\partial \Omega$. Let $u\in H$ be as in the statement. By integration by parts, $\bar{u}$ satisfies \eqref{wkk} and $u \equiv \bar{u}$ in $\Omega$, since weak solutions of \eqref{wkk} are unique.  Finally, if $f \in C^{\alpha}(\Omega)$, by Schauder estimates \cite[Thm. 6.3.2]{krylov:1996}, $\bar v\in C^{2, \alpha} (\overline{\Omega}),$ $u \in C^{4, \alpha} (\overline{\Omega})$ for some $\alpha\in(0,1),$ and \eqref{ret:sch} holds.

Now without loss of generality assume that $\Omega$ is a hyperrectangle of the form $\Omega=\prod_{i=1}^N[0,l_i]$ for some $l_i>0,$ $i=1,\ldots, N$. 
Let $\bar{u}, \bar{v}\in H^1_0 (\Omega)$ be weak solutions of \eqref{sys}. 
Using odd reflections and the Dirichlet boundary conditions we extend $\bar u$ and $\bar v$ to $\tilde \Omega=\prod_{i=1}^N[-l_i,2l_i]$ and obtain
weak solutions of \eqref{sys} (with the odd extension of $f$ as well) defined in $\tilde \Omega$.
Then interior regularity \cite[Theorem 1 Sec 4 Ch 9]{krylov:2008} implies that for any $\overline{\Omega_1} \subset \Omega_2 \subset \tilde \Omega$
\begin{equation}\label{itr}
\|\bar u\|_{W^{4,p}(\Omega_1)}\leq C(\|f\|_{L^\infty(\Omega_2)} + \|\bar u\|_{L^2(\Omega_2)}) \,.
\end{equation}
Note that we replaced $L^p$  by $L^2$ on the right hand side, which can be done using Sobolev embeddings and iteration in \eqref{itr}. Also, $\|\bar u\|_{L^2(\tilde \Omega)} \leq C \|f\|_{L^2(\tilde \Omega)}$ by testing \eqref{sys}
by $\bar v$ and $\bar u$ respectively and using standard estimates. 

As before, $u\equiv \bar u$ in $\Omega$ by integration by parts and uniqueness of weak solutions. 
Finally, \eqref{ret:sch} follows analogously from interior Schauder estimates \cite[Thm. 7.11]{krylov:1996}.

\end{proof}

\begin{lemma}\label{l:reg:2}
Let $\Omega$ be a smooth domain, $\beta, \gamma > 0$, and let $u \in  H \cap L^{\infty}(\Omega)$ be a weak solution of $\gamma \Delta^2 u - \beta \Delta u = u - u^3$ in $H$. For every $p>1$ there is $\gamma_0=\gamma_0(\beta, p, \Omega)>0$ and $C = C(\beta, p, \Omega)$ such that if $\gamma\in(0,\gamma_0)$ then
\begin{align}\label{reg:est}
\|u\|_{W^{6, p}(\Omega)} \leq C ( 1 + \|u\|^9_{L^{\infty}(\Omega)}).
\end{align}
\end{lemma}
\begin{proof} 
In this proof, $C$ denotes different positive constants which depend on $\Omega, \beta$, and $p$, but are \emph{independent} of $\gamma$. 

Since $u \in L^\infty (\Omega)$ we have by bootstrap and Lemma \ref{l:reg} that $u\in C^{4, \alpha}(\overline{\Omega})$ with $\Delta u \in C^0(\overline{\Omega})$. Then 
$(u, v)$ with
$v=-\gamma\Delta u$ and $f = u - u^3$
solves \eqref{sys} in the classical sense. By \cite[Ch. 8. Sec. 5. Thm 6]{krylov:2008}, there is $\gamma_0=\gamma_0(\beta, p, \Omega)>0$ such that, for every $\gamma\in(0,\gamma_0)$,
\begin{align*}
\|v\|_{W^{2, p}(\Omega)} &\leq C (1 + \|u\|^3_{L^{\infty}(\Omega)}) \,.
\end{align*}
On the other hand, $-\beta \Delta u = u - u^3 - \gamma \Delta^2 u = u - u^3 + \Delta v$ and by \cite[Lemma 9.17]{gilbarg:2001} for every $p\in(0,\infty)$ we have
\begin{equation}\label{ueq}
 \|u\|_{W^{2, p}(\Omega)} \leq C \|u-u^3 + \Delta v\|_{L^p(\Omega)} \leq C (1 + \|u\|^3_{L^{\infty}(\Omega)})\,.
\end{equation}
In particular, $\|\Delta u\|_{L^p (\Omega)} \leq C (1 + \|u\|^3_{L^{\infty}(\Omega)}).$ Set $w=\Delta u.$ Since $u\in C^{4,\alpha}(\overline{\Omega}),$ we have that $w$ is a weak solution of  
\begin{align*}
\gamma \Delta^2 w - \beta \Delta w = f_1 := \Delta (u-u^3) =  w - 6u|\nabla u|^2 - 3u^2 w \quad \text{ in }\Omega
\end{align*}
with $w = 0$ on $\partial \Omega$. Moreover, $\Delta w = \Delta^2 u = \frac{1}{\gamma} (u - u^3 + \beta \Delta u) = 0$ on $\partial \Omega$, since 
$u \in  C^{4, \alpha}(\overline{\Omega})$ and $u = \Delta u = 0$. Observe that $\|f_1\|_{L^p (\Omega)} \leq C (1 + \|u\|^7_{L^{\infty}(\Omega)}).$ Then we can repeat the argument for $w$ and $f_1$ instead of $u$ and $f$ respectively to obtain that $w\in C^{4,\alpha}(\overline{\Omega})$ and 
\begin{equation}\label{weq}
\|w\|_{W^{2, p}(\Omega)} \leq C(1 + \|u\|^7_{L^{\infty}(\Omega)}) \,.
\end{equation}
Then $\xi := \Delta w\in C^{2, \alpha} (\overline{\Omega}),$ $\|\xi\|_{L^p (\Omega)} \leq C (1 + \|u\|^7_{L^{\infty}(\Omega)}),$ and $\xi$ is a weak solution of $\gamma \Delta^2 \xi - \beta \Delta \xi = f_2 := \Delta f_1$ and $\xi = 0$ on $\partial \Omega$. Additionally, $\Delta \xi = \Delta^2 w = \frac{1}{\gamma} (f_1 + \beta \Delta w)= 0$ on $\partial \Omega,$ where we  used $u = \Delta u = w = \Delta w$ on $\partial \Omega$. Here we are fundamentally using that $f_1$ vanishes on $\partial \Omega$.  Note that $\|f_2\|_{L^p (\Omega)} \leq C (1 + \|u\|^9_{L^{\infty}(\Omega)})$. Thus we can iterate the above procedure one more time with $u$ and $f$ replaced by $\xi$ and $f_2$ respectively, to obtain
\begin{equation}\label{xieq}
\|\xi\|_{W^{2, p} (\Omega)} \leq C(1 + \|u\|^9_{L^{\infty}(\Omega)}) \,.
\end{equation}
 Note that we cannot iterate anymore since $f_2$ is not vanishing on 
 $\partial \Omega$, and therefore we cannot obtain boundary conditions for higher derivatives. Finally, \eqref{reg:est} follows by \eqref{ueq}, \eqref{weq}, \eqref{xieq}, and \cite[Ch.9 Sec. 2 Thm 3]{krylov:2008}.
 
\end{proof}

\section{A priori bounds}\label{a:priori:section}
For any $\beta > 0$ we fix the following constants for the rest of the paper
\begin{equation}\label{beta:constants}
\begin{aligned}
C_\beta:=\sqrt{\frac{4+\beta^2}{4}},\quad M_\beta:=\frac{1}{\beta^2}\bigg(\frac{4+\beta^2}{3}\bigg)^{\frac{3}{2}},\quad K_0:=\sqrt{\frac{8}{\sqrt{27}-2}}\approx 1.58.
\end{aligned}
\end{equation}
\begin{remark}\label{rmk:csq} We now state some facts that will be useful later in the proofs.
\begin{itemize}
 \item[i)] $C_\beta>1$ is the unique positive root of $h(s)=\frac{4}{\beta^2}(s-s^3)+s$ (cf. Lemma \ref{bounds:lemma}). In particular, $h>0$ in $(0,C_\beta).$
 \item[ii)] $M_\beta\geq 1$ is the maximum value of $h$ in $(0,\infty)$
 \item[iii)] If $\beta \geq \sqrt{8}$ then $h$ is increasing on $(0, 1)$.
 \item[iv)] $M_\beta\leq C_\beta$ if $\beta\geq K_0.$
\end{itemize}
\end{remark}

\begin{lemma}\label{u:bound:lemma}
Let $\Omega\subset \mathbb R^N$ be a bounded domain, $\beta>0,$ and let $u$ be a classical solution of \eqref{efk:eq}.
\begin{enumerate}
\item [i)] If $u$ is nonnegative in $\Omega$, then $\|u\|_{L^\infty(\Omega)}\leq M_\beta.$
\item [ii)] If $\beta\geq \sqrt{8}$ and $u$ is nonnegative in $\Omega$, then $\|u\|_{L^\infty(\Omega)}\leq 1.$
\item [iii)] If $\beta \geq \sqrt{8}$ and $\|u\|_{L^\infty(\Omega)}< \sqrt{\frac{\beta^2}{2}+1},$ then $\|u\|_{L^\infty(\Omega)}\leq 1.$
\end{enumerate}
\end{lemma}
\begin{proof}
We use the notation of Lemma \ref{bounds:lemma} with $f(s)=s-s^3.$ We prove claim $iii)$ first.  Assume without loss of generality that $\|u\|_{L^\infty(\Omega)}=\overline{u}< \sqrt{\frac{\beta^2}{2}+1}.$ By direct computation $g(s)<-s$ for $s\in(-\sqrt{\frac{\beta^2}{2}+1}\,,\,0)$.  Therefore 
\begin{align}\label{u:bound:eq}
\max\limits_{s\in [-\overline{u},0]}g(s)<\max\limits_{s\in [-\overline{u},0]}-s = \overline{u}. 
\end{align}
Moreover, by Lemma \ref{bounds:lemma}, $\overline{u}\leq \max\limits_{[-\underline{u},\overline{u}]}g\leq \max\limits_{[-\overline{u},\overline{u}]}g.$  Then, in virtue of \eqref{u:bound:eq}, we get $\overline{u}\leq\max\limits_{[0,\overline{u}]}g.$ It follows from Lemma \ref{bounds:lemma} and Remark \ref{rmk:csq} iii) that $\overline{u}\leq\max\limits_{[0,1]}g=1$ as claimed.

Now, let $\beta>0$ and assume $\underline{u}\geq 0$. By Lemma \ref{bounds:lemma} and Remark \ref{rmk:csq}, $\overline{u}\leq \max\limits_{[0,\infty]}g
= M_\beta$ and the first claim follows. If $\beta\geq \sqrt{8},$ then $g$ is nondecreasing in $[0,1]$ and, by Lemma \ref{bounds:lemma}, $\overline{u}\leq \max\limits_{[0,1]} \, g=g(1)=1.$ 

\end{proof}

\section{Bounds for the global minimizer}\label{bounds:section}

Recall the definition of $K_0$, $M_\beta$, and $C_\beta$ given in \eqref{beta:constants} and of $H$ given in \eqref{H:def}.

\begin{proposition}\label{bound:global:minimizers}
  Let $v$ be a global minimizer of \eqref{efk:eq} in $H$. Then 
  \begin{align*}
\|v\|_{L^{\infty}(\Omega)}\leq C_\beta\ \text{ if }\beta\geq K_0\quad \text{ and }\quad \|v\|_{L^{\infty}(\Omega)}\leq 1\ \text{ if }\beta \geq \sqrt{8}.
\end{align*}
\end{proposition}
When $\beta\in(0,\sqrt{8})$ we expect that global minimizers are not bounded by 1 in big enough domains, see Figure \ref{figure1} and Figure \ref{figure2}.  

\begin{proof}
Assume $\beta\geq K_0$, let $J_\beta$ be as in \eqref{J0:definition}, and let $v$ denote a global minimizer in $H,$ i.e. $J_\beta(v)\leq J_\beta(w)$ for all $w\in H.$ We show first that $v\in L^\infty(\Omega)$. Let
\begin{align*}
f:\R\to\R,\ \  f(s):=\begin{cases}
	s-s^3& \text{ if }s\in [-C_\beta,C_\beta],\\
	\sign(s) (C_\beta-C_\beta^3) = -\frac{\beta^2}{4} \sign(s)C_\beta & \text{ otherwise },
        \end{cases}
\end{align*}
and let 
\begin{align}\label{J:plus:functional}
 J:H\to \R;\qquad J(u)&:=\int_\Omega \frac{|\Delta u|^2}{2} + \beta \frac{|\nabla u|^2}{2} - F(u) \, dx,
\end{align}
where $F(s):=\int_0^s f(t)\, dt.$  Let $u$ be the global minimizer of $J$ in $H.$ By Lemma \ref{l:reg}, $u$ 
is a classical solution of 
\begin{equation}\label{gamma:equation:plus}
\begin{aligned}
 \Delta^2 u -\beta \Delta u &= f(u)&&\quad \text{ in }\Omega,\\
 u=\Delta u&=0 &&\quad \text{ on }\partial \Omega.
\end{aligned}
\end{equation}
Using Lemma \ref{bounds:lemma} (and its notation), since $g\geq 0$ in $(0,\infty)$, we have $\underline{u}\geq\min\limits_{[\underline{u},\overline{u}]}g=\min\limits_{[\underline{u},0]}g$,
and consequently $\underline{u}\geq \min\limits_{[-1,0]}g\geq-M_\beta$. If $\beta\geq \sqrt{8}$, we additionally have
$\underline{u} \geq \min\limits_{[-1,0]}g=-1$.

Replacing $u$ by $-u$ and noting that $\beta \geq K_0$ we conclude $\|u\|_{L^{\infty}(\Omega)}\leq M_\beta \leq C_{\beta}$ 
and $\|u\|_{L^{\infty}(\Omega)}\leq 1$ if $\beta \geq \sqrt{8}$. In particular $f(u(x)) = u(x) - u^3(x)$ for $x\in\Omega$, and therefore 
\begin{equation*}
J(u)= J_\beta(u) \geq J_\beta(v) \geq J(v)\,,
\end{equation*}
where the last inequality is strict if $\|v\|_{L^{\infty}(\Omega)}>C_\beta$, a contradiction to the minimality of $u$. Thus $v\in L^\infty(\Omega)$ and, by Lemma \ref{l:reg}, $v$ is a classical solution of \eqref{gamma:equation:plus}.  If $\beta \geq \sqrt{8}$ and $\|v\|_{L^{\infty}(\Omega)}>1$, then, by Lemma \ref{u:bound:lemma} part $iii),$ $\|v\|_{L^{\infty}(\Omega)}\geq \sqrt{\frac{\beta^2}{2}+1}>C_\beta$ and we obtain a contradiction as above.

\end{proof}

\section{Existence of positive solutions}\label{positive:solutions:section}

Recall the definition of $K_0$, $M_\beta$, and $C_\beta$ given in \eqref{beta:constants}.

\begin{lemma}\label{u:bound:lemma:plus}
Let $\Omega\subset \mathbb R^N$ be a bounded domain, $\beta>0,$ 
\begin{align}\label{f:definition}
f:\R\to\R; \qquad f(s):=\begin{cases}
	-\frac{\beta^2}{4}s &\quad \text{ if }s<0,\\
	s-s^3& \quad \text{ if }s\in [0,C_\beta],\\
	C_\beta-C_\beta^3 = -\frac{\beta^2}{4}C_{\beta} &\quad \text{ if }s>C_\beta,
        \end{cases}
\end{align}
and let $u$ be a classical solution of \eqref{gamma:equation:plus} for this choice of $f$. Then $0\leq u \leq M_\beta.$ Moreover, if $\beta\geq \sqrt{8},$ then $u\leq 1.$ In particular, if $\beta\geq K_0$, then $u$ is a classical solution of \eqref{efk:eq}.
\end{lemma}

\begin{proof}
By Lemma \ref{bounds:lemma} (and using the same notation) $\underline{u}\geq \min\limits_{\R} g \geq 0$ by the definition of $C_\beta,$ that is, $u\geq 0$ in $\Omega.$ On the other hand, again by Lemma \ref{bounds:lemma} we have that $\overline{u}\leq \max\limits_{[0,\overline{u}]} \, g=\max\limits_{[0,1]} \, g=M_\beta.$ 

If $\beta\geq \sqrt{8},$ then $g$ is nondecreasing in $[0,1]$, and therefore $\overline{u}\leq \max\limits_{[0,\overline{u}]}g$ implies, by Lemma \ref{bounds:lemma}, that $\overline{u}\leq \max\limits_{[0,1]} \, g=g(1)=1.$ Finally, if $\beta\geq K_0,$ then $M_\beta\leq C_\beta,$ and thus $f(u(x))=u(x)-u(x)^3$ for all $x\in\Omega,$ that is, $u$ solves \eqref{efk:eq}.

\end{proof}

\begin{theorem}\label{existence:positive:solution}
Let $\beta>0$ and $\Omega\subset \mathbb \R^N$ with $N\geq 1$ be a smooth bounded domain or a hyperrectangle. If 
$\lambda_1^2 + \beta\lambda_1 \geq 1$, then $u\equiv 0$ is the unique weak solution of \eqref{efk:eq}.  If $\lambda_1^2+\beta\lambda_1 < 1$, then for $\beta\geq K_0$ there is a positive classical solution $u$ of \eqref{efk:eq} such that $\|u\|_{L^\infty(\Omega)} \leq M_\beta$ and $\Delta u < \frac{\beta}{2}u$ in $\Omega$. Additionally, if $\beta\geq \sqrt{8}$ then $\|u\|_{L^\infty(\Omega)}\leq 1$.
\end{theorem}
\begin{proof}
Let $\lambda_1^2+\beta\lambda_1 \geq 1$ and assume by contradiction that there is a nontrivial weak solution $u \in H$ of \eqref{efk:eq}. Testing equation \eqref{efk:eq} with $u$ yields
 \begin{align*}
  0=\int_\Omega |\Delta u|^2+\beta|\nabla u|^2+u^4-u^2\, dx> (\lambda_1^2+\beta\lambda_1-1) \int_\Omega u^2 \, dx\geq 0,
 \end{align*}
by the Poincar\' e inequality, a contradiction.

Now, assume that $\lambda_1^2+\beta\lambda_1 < 1$, let $f$ be as in \eqref{f:definition}, and $J$, $F$ as in \eqref{J:plus:functional} with this choice of $f$. Note that $F(s)\leq \frac{1}{4}$ for all $s\in\R.$ Thus $ J(u)\geq -\frac{|\Omega|}{4}$ for all $u\in H.$  Standard arguments show that $J$ attains a global minimizer $u$ in $H$ and that $u$ is a weak solution of $\Delta^2 u - \beta \Delta u = f(u)$ in $\Omega$ with Navier boundary conditions. Observe that $u\not\equiv 0$, since $\lambda_1^2+\beta\lambda_1 < 1$ implies 
$J(\delta \varphi_1) < 0$ for sufficiently small $\delta > 0$, where $\varphi_1$ is the first Dirichlet eigenfunction of the Laplacian in $\Omega$.

Arguing as in Proposition \ref{bound:global:minimizers} we have that $u\in L^\infty(\Omega)$, and therefore $u\in C^4(\Omega)\bigcap C(\overline{\Omega})$ and $\Delta u\in C_0(\Omega)$, by Lemma \ref{l:reg}. Then $u$ is a classical solution of \eqref{efk:eq} satisfying $0\leq u\leq M_\beta$ if $\beta \geq K_0$, and $0\leq u\leq1$ if $\beta\geq \sqrt{8}$, by Lemma \ref{u:bound:lemma:plus}.

The strict positivity of $u$ (recall $u \not \equiv 0$) and $-\Delta u +\frac{\beta}{2}u$ is a consequence of the maximum principle for second order equations and the following decomposition into a second order system
\begin{equation*}
\begin{aligned}
 -\Delta u +\frac{\beta}{2}u= w,\  -\Delta w +\frac{\beta}{2}w= (1+\frac{\beta^2}{4})u-u^3 \ \text{ in }\Omega,\quad u=w=0 \ \text{ on }\partial\Omega,
\end{aligned}
\end{equation*}
where $(1+\frac{\beta^2}{4})u-u^3\geq 0$ since $0\leq u\leq M_\beta\leq C_\beta$ for $\beta \geq K_0$.
\end{proof}

\section{Stability of positive solutions}\label{section:stability}

In this section we prove the following result.

\begin{theorem}\label{stability:thm}
Let $\partial \Omega$ be of class $C^{1,1}$ and $\beta\geq \sqrt{8}$. Then any positive solution of \eqref{efk:eq} is strictly stable.
\end{theorem}

The proof of Theorem \ref{stability:thm} is an easy consequence of the following.

\begin{proposition}\label{first:eigenvalue:is:zero} Assume that $\beta\geq \sqrt{8}$, and let $u$ be a positive solution of \eqref{efk:eq}. Then,
\begin{align*} 
\mu_1=\inf_{v\in H\backslash\{0\}}\frac{\int_\Omega |\Delta v|^2+\beta |\nabla v|^2 +(u^2-1)v^2\, dx}{\int_\Omega v^2 \, dx}=0 \,.
\end{align*}
\end{proposition}

Indeed, assume for a moment that Proposition \ref{first:eigenvalue:is:zero} holds. Then we have.

\begin{proof}[Theorem \ref{stability:thm}]
Let ${\cal H}:=\{v\in H\::\: \|v\|_{L^2(\Omega)}=1\}$. By Proposition \ref{first:eigenvalue:is:zero},
\begin{align*}
 \inf_{v\in {\cal H}} \int_\Omega |\Delta v|^2+\beta |\nabla v|^2 +&(3u^2-1)v^2\, dx = \mu_1+2\inf_{v\in {\cal H}} \int_\Omega u^2v^2\, dx>0
\end{align*}
and Theorem \ref{stability:thm} follows. 

\end{proof}

The proof of Proposition \ref{first:eigenvalue:is:zero} uses the following result adjusted to our situation.

\begin{theorem}\label{sweers:thms} 
Let $L:=\begin{pmatrix} -\Delta & 0 \\ 0 & -\Delta \end{pmatrix}$ and $M$ be a $2\times2$ continuous matrix such that
\begin{enumerate}
 \item[1)] $-M$ is \emph{essentially positive}, that is, $-M_{12} \geq 0$ and $-M_{21}\geq 0$ in $\Omega$.
 \item[2)] $M$ is \emph{fully coupled}, that is, $M_{1,2}\not\equiv 0$ and $M_{2,1}\not\equiv 0$ in $\Omega$.
 \item[3)] there is a \emph{positive strict supersolution} $\phi$ of $L+M$, i.e., a function $\phi>0$ such that $(L+M)\phi > 0$ in $\Omega$.
\end{enumerate}

Then there are $\widetilde v, \widetilde w\in W^{2,N}_{loc}(\Omega)\bigcap C_0(\Omega)$ unique (up to normalization) positive functions such that 
\begin{align}\label{p1}
 (L+M)\begin{pmatrix} \widetilde v \\ \widetilde w \end{pmatrix}=\lambda_0 \begin{pmatrix} \widetilde v \\ \widetilde w \end{pmatrix},
\end{align}
where $\lambda_0>0$ is the smallest eigenvalue (smallest real part) of \eqref{p1}.

Moreover, there are  positive functions $v, w\in W^{2,N}_{loc}(\Omega)\bigcap C_0(\Omega)$ unique up to normalization such that 
\begin{align}\label{p2}
 (L+M)\begin{pmatrix} v \\ w \end{pmatrix}=\lambda_B \,B\, \begin{pmatrix} v \\ w \end{pmatrix},
\end{align}
where $\lambda_B>0$ is the smallest eigenvalue (smallest real part) 
of \eqref{p2} and 
$B \not \equiv 0$ is a matrix with $B_{ij} \in C(\bar{\Omega})$ and $B_{ij} \geq 0$.
\end{theorem}

Theorem \ref{sweers:thms} is a particular case of \cite[Theorem 1.1]{sweers:1992} and \cite[Theorem 5.1]{mitidieri:1995}. In \cite[Theorem 5.1]{mitidieri:1995} the result is formulated for matrices $B \in C_0(\bar{\Omega})$, but the same proof applies for $B \in C(\bar{\Omega})$.

\begin{proof}[Proposition \ref{first:eigenvalue:is:zero}]
Let
\begin{align}\label{matrix:notation}
L:=\begin{pmatrix} -\Delta & 0 \\ 0 & -\Delta \end{pmatrix},\qquad \widetilde M:=\begin{pmatrix} \frac{\beta}{2} & -1 \\ u^2-1-\frac{\beta^2}{4} & \frac{\beta}{2} \end{pmatrix},\qquad B:=\begin{pmatrix} 0 & 0 \\ 1 & 0 \end{pmatrix},
\end{align}
and
\begin{align*}
M:=\widetilde M + \frac{\beta^2}{4}B=\begin{pmatrix} \frac{\beta}{2} & -1 \\ u^2-1 & \frac{\beta}{2} \end{pmatrix}.
\end{align*}

Let $\varphi_1>0$ be the first Dirichlet eigenfunction of the Laplacian in $\Omega$, i.e.,
\begin{equation}\label{first:eigenfunction}
\begin{aligned}
 -\Delta \varphi_1 = \lambda_1 \varphi_1 \quad \text{ in }\Omega,\qquad \varphi_1 = 0\quad \text{ on }\partial\Omega.
\end{aligned}
\end{equation}
Note that $-M$ is fully coupled, $M$ is essentially positive by Lemma \ref{u:bound:lemma} ii), and, for $\beta\geq \sqrt{8}$,
the function $\phi=\begin{pmatrix} \varphi_1 \\ \varphi_1 \end{pmatrix}$ is a positive strict supersolution of $L+M$, because $\lambda_1+\frac{\beta}{2}-1>0$ and $\lambda_1+u^2-1+\frac{\beta}{2}>0$.  Then, by Theorem \ref{sweers:thms}, there are unique (up to normalization) positive functions $v, w\in W^{2,N}_{loc}(\Omega)\bigcap C_0(\Omega))$ such that
\begin{align*}
 (L+\widetilde M)\begin{pmatrix} v \\ w \end{pmatrix}=(\lambda_B-\frac{\beta^2}{4}) \,B\, \begin{pmatrix} v \\ w \end{pmatrix},
\end{align*}
where $\lambda_B>0$ is the smallest eigenvalue.

Since $\partial\Omega$ is $C^{1,1}$ standard regularity arguments imply that $v\in C^4(\Omega) \cap C_0(\bar{\Omega})$, 
$\Delta v \in C_0(\bar{\Omega})$ and then
\begin{align}\label{6:5}
 \Delta^2 v - \beta \Delta v +(u^2-1)v& = (\lambda_B-\frac{\beta^2}{4}) v = \mu_1 v \qquad \text{ in }\Omega,
 \end{align}
where $\mu_1:=\lambda_B-\frac{\beta^2}{4}$ is the smallest eigenvalue of \eqref{6:5}.  By multiplying this equation by $u>0$ and integrating by parts we get that
\begin{align*}
 0=\mu_1\int_{\Omega}\, v\, u\,dx
\end{align*}
since $u$ is a solution of \eqref{efk:eq}. Therefore $\mu_1=0,$ because $u$ and $v$ are positive. This ends the proof by the variational characterization of $\mu_1.$

\end{proof}

\begin{remark}
 Note that \emph{without} the assumption $\beta\geq \sqrt{8}$ the result still holds for any solution $u$ such that $0< u\leq 1$ in $\Omega$, $\lambda_1+\frac{\beta}{2}-1>0$, and $\lambda_1+u^2-1+\frac{\beta}{2}>0.$ 
\end{remark}

\begin{remark}\label{tobias}
Uniqueness of positive solutions of \eqref{efk:eq} when $\beta\geq \sqrt{8}$ can be proved using Proposition \ref{first:eigenvalue:is:zero} as follows. Let $u$ and $v$ denote two positive solutions of \eqref{efk:eq}. Then $w:=u-v$ solves  $\Delta^2 w - \beta \Delta w + (u^2-1) w = -(uv+v^2)w$ in $\Omega$ and $w=\Delta w=0$ on $\partial \Omega$. By testing with $w$ we have
\begin{align*}
\int_\Omega |\Delta w|^2+\beta |\nabla w|^2 +(u^2-1)w^2\, dx = \int_\Omega (-uv-v^2)w^2\ dx \leq 0.
\end{align*}
But then Proposition \ref{first:eigenvalue:is:zero} yields that $w\equiv 0$. 
\end{remark}

\section{Radial symmetry of stable solutions}\label{radial:symmetry:section}

\begin{proof}[Theorem \ref{theorem:radial:minimizers}]
Let $L$ and $B$ be as in \eqref{matrix:notation} and let
\begin{align*}
\widetilde Q:=\begin{pmatrix} \frac{\beta}{2} & -1 \\ 3u^2-1-\frac{\beta^2}{4} & \frac{\beta}{2} \end{pmatrix}.
\end{align*}
Note that 
\begin{align*}
Q:=\widetilde Q + \left(\frac{\beta^2}{4}-2\right)B=\begin{pmatrix} \frac{\beta}{2} & -1 \\ 3(u^2-1) & \frac{\beta}{2} \end{pmatrix}
\end{align*}
is a fully coupled matrix and $-Q$ is essentially positive (see Theorem \ref{sweers:thms}). Moreover, since $\beta > \sqrt{12}-2\lambda_1$ there is $\delta>0$ such that $1/(\lambda_1+\frac{\beta}{2})<\delta<(\lambda_1+\frac{\beta}{2})/3.$ For such $\delta$, if $\varphi_1$ is as in \eqref{first:eigenfunction}, we have 
\begin{align*}
(L+Q)\begin{pmatrix} \delta\varphi_1 \\ \varphi_1 \end{pmatrix}=\begin{pmatrix} (\delta\lambda_1 + \frac{\beta}{2}\delta-1)\varphi_1 \\ 
 (\lambda_1 + 3(u^2-1)\delta+\frac{\beta}{2})\varphi_1 \end{pmatrix}>0\,,
\end{align*}
and therefore $L+Q$ has a positive strict supersolution.  Therefore, by Theorem \ref{sweers:thms} there is $\lambda_B>0$ and 
positive functions  $v,w\in W_{loc}^{2,N}(\Omega)\bigcap C_0(\Omega)$ such that
\begin{align*}
 (L+\widetilde Q)\begin{pmatrix} v \\ w \end{pmatrix}=\mu\, B \begin{pmatrix} v \\ w \end{pmatrix}, \qquad \text{ with }\mu=\lambda_B-\frac{\beta^2}{4}+2.
\end{align*}

By standard regularity arguments,  $v\in C^{4}(\Omega)\bigcap C_0(\Omega)$ and $v$ solves
\begin{align*}
 \Delta^2 v - \beta \Delta v +(3u^2-1)v& =\mu v \qquad \text{ in }\Omega,
\end{align*}
that is, $v$ is the first eigenfunction and $\mu$ is a simple (first) eigenvalue. Stability of $u$ implies $\mu\geq 0$.  We now argue by contradiction, assume that $u$ is not radially symmetric.  Then there is a nontrivial angular derivative $u_\theta=\partial_\theta u\not\equiv 0$. In particular, $u_\theta$ must change sign in $\Omega$ and $\Delta^2 u_\theta - \beta \Delta u_\theta +(3u^2-1)u_\theta = 0$ in $\Omega,$ since $\partial_{\theta}$ and $-\Delta$ commute and $u$ is a solution of \eqref{efk:eq}. This implies that $u_\theta$ is a sign-changing eigenfunction associated to the zero eigenvalue, but this contradicts the fact that $\mu\geq 0$ is the simple first eigenvalue. Therefore $u$ is a radial function.

\end{proof}

\begin{remark}\label{beta:res}
Note that Theorem \ref{theorem:radial:minimizers} holds true for any $\lambda_1$ if $\beta\ge \sqrt{12}.$ On the other hand, if $u$ is a nontrivial solution, then, by Theorem \ref{existence:positive:solution}, we have that $\lambda_1^2+\beta\lambda_1 < 1$. Combined with $\beta> \sqrt{12}-2\lambda_{1}$, we obtain that the infimum of  
$\beta$'s satisfying these inequalities is $\sqrt{8}$ with corresponding 
$\lambda_1 = \sqrt{3} - \sqrt{2}$.
\end{remark}

\section{Positivity of global radial minimizers}\label{positivity:radial:section}

Before we prove 
Theorem \ref{positivity:minimizers}, we introduce some notation. Let 
 \begin{align*}
 I:=\left\{
	\begin{array}{ll}
		[0,R) & \mbox{ if } \Omega=B_R,\\
		(R_0,R) & \mbox{ if } \Omega=B_R \backslash B_{R_0},
	\end{array}
\right.
\end{align*}
for some $R>R_0>0$, where $B_r$ denotes the open ball of radius $r$ centered
at the origin.  To simplify the presentation we abuse a little bit the notation and also denote $u(r)=u(|x|).$ 
Recall that $H_r=\{u\in H\::\: u\text{ is radially symmetric}\}.$

\begin{proof}[Theorem \ref{positivity:minimizers}]
Suppose first that $\Omega=B_R(0)$ and for a contradiction, assume that $u$ changes sign. Then either $u$ or $-u$ has a positive local 
maximum in $(0,R)$. Without loss of generality, assume there is $\eta\in(0,R)$ such that $1\geq M:=u(\eta)=\max\limits_{[0,R]}u>0.$ Let $v:\overline{I}\to\mathbb R$ be given by
 \begin{align*}
  v:=\left\{
	\begin{array}{ll}
		\frac{1-M}{1+M}(M-u)+M & \mbox{ in } [0,\eta],\\
		u & \mbox{ in } (\eta,R]\,.
	\end{array}
\right.
 \end{align*}
Note that $v$ is just a rescaled reflection of $u$ with respect to $u=M$ in $[0,\eta]$. Since $v'(\eta)=u'(\eta)=v'(0)=u'(0)=0$ we have that $v\in C^1(I)$. Also
\begin{equation}\label{smaller:derivatives}
\begin{aligned}
&\qquad|v'|\leq |u'|\quad  \text{ in }(0,R) \quad  \text{ and }\quad |v''|\leq |u''| \quad  \text{ in }(0,R) \setminus \{\eta\},\\
&\text{with strict inequalities in $(0, \eta)$, whenever $u'\neq 0$ and $u''\neq 0$.}
\end{aligned}
\end{equation}
 Furthermore in $(0,\eta)$ one has $0\leq v\leq \frac{1-M}{1+M}(M+1)+M=1$ and 
\begin{equation*}
\begin{aligned}
v-u&= \left(\frac{1-M}{1+M}+1\right)(M-u)\geq 0 \,,\\
v+u&= \left(-\frac{1-M}{1+M}+1\right)u+ \left(\frac{1-M}{1+M}+1\right)M = \left(\frac{2M}{1+M}\right)(u+1)\geq 0 \,.
\end{aligned}
\end{equation*}
  Then $v^2-u^2=(v-u)(v+u)\geq 0,$ thus $|1-v^2|=1-v^2\leq 1-u^2=|1-u^2|$ in $[0,R],$ and then
\begin{align}\label{smaller:nonlinear:part}
 \int_I (v^2-1)^2\ dx \leq \int_I (u^2-1)^2\ dx.
\end{align}
By \eqref{smaller:derivatives} and \eqref{smaller:nonlinear:part} one has $J_\beta(v)<J_\beta(u)$, a contradiction to the minimality of $u$ and thus $u$ does not change sign in $\Omega.$

The proof for the annulus $\Omega=B_R(0)\backslash B_{R_0}(0)$ for some $R>R_0>0$ is similar. If $u$ changes sign, then there are $\eta,\mu\in I$ such that 
\begin{align*}
M:=u(\eta)=\max_{[0,R]}u >0\qquad \text{ and }\qquad m:=-u(\mu)=-\min_{[0,R]}u >0.
\end{align*}
Without loss of generality, assume that $\eta<\mu.$ Let $v:\overline{I}\to\mathbb R$ be given by
 \begin{align*}
  v:=\left\{
	\begin{array}{ll}
		u & \mbox{ in } [R_0,\eta],\\
		\frac{m-M}{m+M}(M-u)+M & \mbox{ in } (\eta,\mu),\\
		-u & \mbox{ in } [\mu,R]\,.
	\end{array}
\right.
\end{align*}
Then in $(\eta,\mu)$
\begin{align*}
 v&\leq \frac{m-M}{m+M}(M+m)+M=m\leq 1\qquad \text{ if }m>M,\\
 v&=\frac{M-m}{m+M}(u-M)+M\leq M\leq 1\qquad \text{ if }m\leq M,\\
 v-u&=\left(\frac{M-m}{m+M}-1\right)u+\left(\frac{m-M}{m+M}+1\right)M=\frac{2m}{m+M}(-u+M)\geq 0,\\
 v+u&=\left(\frac{M-m}{m+M}+1\right)u+\left(\frac{m-M}{m+M}+1\right)M=\frac{2M}{m+M}(u+m)\geq 0\,.
\end{align*}
Therefore \eqref{smaller:derivatives} (with strict inequality on $(\eta, \mu)$) and \eqref{smaller:nonlinear:part} also hold in this case. 
Then $J_\beta(v)<J_\beta(u),$ which contradicts the minimality of $u$ and thus $u$ does not change sign in $\Omega$.

It only remains to prove the assertions about $\partial_r u$.  We prove the case when $\Omega$ is an annulus, since the case when $\Omega$ is a ball can be treated similarly.

Without loss of generality assume $u\geq 0$ in $\Omega.$ Now, suppose by contradiction that $u'$ has more than one change of sign. Then $u'$ has at least three changes of sign: two local maxima and one local minimum. Let $\eta_1\in I$ be such that $u(\eta_1)=\max\limits_{I}u=:M_1$ and $\tilde\eta\in I$ another local maximum. We only prove the case $\eta_1<\tilde \eta,$ the other one is analogous. Let $\mu\in (\eta_1,\tilde\eta)$ be such that $u(\mu)=\min\limits_{[\eta_1,\tilde\eta]}u=:m$ and $\eta_2\in(\mu,\tilde\eta]$ such that $u(\eta_2)=\max\limits_{[\mu,\tilde\eta]}u=:M_2.$ Clearly $m<M_2\leq M_1.$ Define $v:I\to [0,1]$ by
\begin{align*}
  v:=\left\{
	\begin{array}{ll}
		u & \mbox{ in } [R_0,\eta_1]\cup[\eta_2,R],\\
		\frac{M_1-M_2}{M_1-m}(u-M_1)+M_1 & \mbox{ in } (\eta_1,\mu),\\
		M_2 & \mbox{ in } [\mu,\eta_2),
	\end{array}
\right.
\end{align*}
By similar calculations as above it is easy to see that $v\in C^1(I),$ and that $J_\beta(v)<J_\beta(u),$ which contradicts the minimality of $u$ and thus $u'$ only changes sign once in $I$. 

\end{proof}
 
\section{Symmetry of positive solutions}\label{symmetry:section}

\begin{proof}[Proposition \ref{symmetry:theorem}]
 We can write \eqref{efk:eq} as the following system.
 \begin{equation}\label{cooperative:system}
 \begin{aligned}
-\Delta u+\frac{\beta}{2}u&=w&&\qquad \text{ in }\Omega,\\
-\Delta w+\frac{\beta}{2}w&=u-u^3+\frac{\beta^2}{4}u&&\qquad \text{ in }\Omega,\\
u=w&=0&&\qquad \text{ on }\partial\Omega\,.\\
  \end{aligned}
 \end{equation}
By Lemma \ref{u:bound:lemma} and $\beta\geq \sqrt{8}$ one has $\|u\|_{L^\infty(\Omega)}\leq 1$ and it is easy to check that 
a standard moving plane method as in \cite{troy:1981} can be applied to \eqref{cooperative:system}. The regularity of the boundary can be relaxed up to Lipschitz regularity if one uses maximum principles for small domains, see \cite{Berestycki:1991,FoldesPolacik:2009}. 

\end{proof}

\section{Saddle solution}\label{saddle:section}

Recall the definition of $K_0$, $M_\beta$, and $C_\beta$ given in \eqref{beta:constants}.

\begin{proof}[Theorem \ref{saddle:solution:theorem}]
Let $\beta\geq K_0$.  By odd reflection, it suffices to find a positive $u \in C^{4}(\overline{\mathbb R^2_{+}})$ solving 
\begin{equation}\label{equation:positive:quadrant}
\begin{aligned}
\Delta^2 u - \beta \Delta u &= u - u^3 &&\qquad \text{ in }\mathbb R^2_{+},\\
\Delta u = u &= 0 &&\qquad \text{ on }\partial \mathbb R^2_{+},
\end{aligned}
\end{equation}
where $\mathbb R^2_{+}=\{(x_1,x_2)\in\mathbb R^2\::\: x_1>0,x_2>0\}$. Indeed, denote again by $u$ the extension to $\mathbb R^2$ by odd reflections. Since $u = 0$ on $H_1 := \{x : x_2 = 0\}$ one has  $u_{x_1} = u_{x_1 x_1} = u_{x_1 x_1 x_1} = 0$ on $H_1$. Also $\Delta u = 0$ on $H_1$ implies $u_{x_2 x_2} = 0$ on $H_1$, and consequently $u_{x_2x_2x_1} = 0$ on $H_1$. Since $u_{x_1}$, $u_{x_1 x_1}$, $u_{x_2x_2}$,  $u_{x_1 x_1 x_1}$, $u_{x_2x_2x_1}$ are
odd with respect to $H_1$ they are continuous on $H_1$. Moreover, all other partial derivatives (up to third order) are continuous on $H_1$ as well, since they are even functions.  The same procedure applies to $H_2 := \{x : x_1 = 0\}$ and thus $u \in C^3(\R^2)$.

From the equation \eqref{equation:positive:quadrant} we also obtain continuity of $\Delta^2 u$ and $\Delta^2 u = 0$ on 
$\partial \mathbb R^2_{+}$. By a similar reasoning as above
we obtain that the extension is of class $C^4(\mathbb R^2)$ and it is a classical solution of \eqref{efk:eq} with $\Omega = \mathbb R^2$, and 
consequently it is a saddle solution. 

We find a positive solution of \eqref{equation:positive:quadrant} by a limiting procedure using Theorem \ref{existence:positive:solution} with $\Omega_R:=(0,R)^2$ and letting $R\to\infty$.  Note that for $R$ big enough, the solution $u_R$ given by Theorem \ref{existence:positive:solution} satisfies that $0<u_R\leq M_\beta$ in $\Omega_R$, and therefore, by Lemma \ref{l:reg}, there is some $C>0$ independent of $R$ such that
\begin{align}\label{C:definition}
 \|u_R\|_{C^{4,\alpha}(\overline{\Omega}_R)}<C\quad \text{ for all } R > 0.
\end{align}

By the Arzela-Ascoli theorem there is a sequence $R_N \to \infty$ such that $u_{R_N} \to u$ in $C^4(\overline{\mathbb R^2_+})$ for some $u$ satisfying \eqref{equation:positive:quadrant}. We now prove that $u\not\equiv 0$. Indeed, fix
\begin{align}\label{r:definition}
 r>80(1+\beta)\widetilde C^2+160,
\end{align}
where $\widetilde C>0$ is a constant independent of $R$ specified below. We show that
\begin{align}\label{lower:bound}
\|u_R\|_{L^\infty([0,r+2]^2)}\geq \frac{1}{\sqrt{2}}\quad \text{ for all }R>r+3.
\end{align}
Assume by contradiction that there is $R>r+3$ such that 
\begin{align}\label{contradiction:hyp}
\|u_R\|_{L^\infty([0,r+2]^2)}\leq \frac{1}{\sqrt{2}}.
\end{align}
We define the following sets
\begin{align*}
 \omega_1&:=\{x\in\Omega_R\::\: \operatorname{dist}(x,\partial \Omega_R)\leq 1,\ x_1\leq r+1, x_2\leq r+1\},\\
 \omega_2&:=\{x\in\Omega_R\::\: \operatorname{dist}(x,\partial \Omega_R)\geq 1,\ x_1\leq r+1, x_2\leq r+1\},\\
 \omega_3&:=\{x\in\Omega_R\::\: r+1\leq\max\{x_1, x_2\}\leq r+2\},\\
 \omega_4&:=\{x\in\Omega_R\::\:  x_1\geq r+2, \textnormal{ or } x_2\geq r+2\}.
\end{align*}
Note that $\Omega_R=\bigcup_{i=1}^4\omega_i.$ Now, let $\phi_1\in C^2(\overline{\Omega_R})\bigcap C_0(\overline{\Omega_R}),$ $\phi_2\in C^2(\overline{\Omega_R})$ 
such that $0\leq \phi_i\leq 1,$ $\|\phi_i\|_{C^2(\Omega_R)}\leq K$ for $i=1,2$ and some $K>0$ independent of $R,$ and 
\begin{align*}
 \phi_1\equiv 1\quad \text{ in }\omega_2,\qquad \phi_1\equiv 0\quad \text{ in }\omega_4,\qquad \phi_2\equiv 0\quad \text{ in }\omega_2,\qquad \phi_2\equiv 1\quad \text{ in }\omega_4.
\end{align*}
Further, let $\psi\in C^4(\Omega_R)\bigcap C_0(\Omega_R)$ be given by $\psi:=\phi_1+\phi_2u_R.$ Then $\psi\equiv 1$ in $\omega_2,$ $\psi\equiv u_R$ in $\omega_4,$ and there is some $\widetilde C>0,$ depending only on $C$ (from \eqref{C:definition}) and $K$, such that 
$\|\psi\|_{C^2(\Omega_R)}\leq \widetilde C.$  For $i=1,\ldots,4$ let 
\begin{align*}
J_{i}(v):=\int_{\omega_i} \frac{|\Delta v|^2}{2}+\beta \frac{|\nabla v|^2}{2} + \frac{1}{4}(v^2-1)^2\ dx \quad \text{ for $v\in H^2(\Omega_R)\bigcap H_0^1(\Omega_R).$}
\end{align*}
Note that $\sum\limits_{i=1}^4 J_{i}(v) = J_\beta(v)+\frac{1}{4}|\Omega_R|$ for $v\in C^2(\Omega_R),$ here $J_\beta$ is as in \eqref{J0:definition} for $\Omega=\Omega_R.$  Then $\sum\limits_{i=1}^4 J_{i}(\psi)\leq[(\frac{1}{2}+\frac{\beta}{2})\widetilde C^2+1](|\omega_1|+|\omega_3|)+J_{4}(u_R),$ and by \eqref{contradiction:hyp}, $\sum\limits_{i=1}^4 J_{i}(u_R)\geq \frac{|\omega_2|}{16}+J_{4}(u_R).$  Therefore 
\begin{align*}
J_\beta(u_R)-J_\beta(\psi)\geq r\left(\frac{r}{16}-5((1+\beta)\widetilde C^2+2)\right)>0, 
\end{align*}
by \eqref{r:definition}, a contradiction to the minimality of $u_R.$ Therefore \eqref{lower:bound} holds and the maximum principle yields that $u>0$ in $\mathbb R^2_+$ is a solution of \eqref{equation:positive:quadrant}. 

\end{proof}

\section{Bifurcation from a simple eigenvalue}\label{bifurcation:section}

\begin{theorem}\label{bifurcation:beta}
Let $\Omega\subset \mathbb \R^N,$ $N\geq 1$ be a smooth bounded domain or a hyperrectangle. If the first Dirichlet eigenvalue $\lambda_1<1$, then there is $\varepsilon>0$ such that \eqref{efk:eq} admits a positive solution $u_\beta\in C^{4,\alpha}(\Omega)$ for all $\beta\in(\bar\beta-\varepsilon,\bar\beta),$ where
\begin{align}\label{eq:beta}
\bar\beta=\frac{1-\lambda_1^2}{\lambda_1}.
\end{align}
Additionally, if $\bar\beta>\sqrt{8}$ and $\Omega$ is smooth, then \eqref{efk:eq} admits a unique positive solution $u_\beta$ such that $\|u_\beta\|_{L^\infty(\Omega)}\leq 1$ for all $\beta\in[\sqrt{8},\bar\beta)$. 
\end{theorem}

\begin{proof}
Let $X=\{ u \in C^{4,\alpha}(\Omega) \cap C^2(\overline{\Omega}) : u = \Delta u = 0 \textrm{ on } \partial \Omega\}$ and $Y=C^{0,\alpha}(\Omega).$  Consider the operator 
\begin{align*}
G:\R\times X\to Y;\quad G(\beta,u):= \Delta^2 u -\beta \Delta u - u + u^3.
\end{align*}
Then we have $G(\beta,0)=0$ for all $\beta$. Moreover $u\in  X$ solves \eqref{efk:eq} if and only if $G(\beta,u)=0.$ We consider the partial derivative
\begin{align*}
 \partial_u G:(0,\infty)\times X\to {\cal L}(X,Y),\quad \partial_u G(\gamma,u)[v]= \Delta^2v-\beta\Delta v -v+3u^2v.
\end{align*}
For $\beta>0$ we set $A_\beta := \partial_u G(\beta,0)$, i.e., $A_\beta v=\Delta^2v-\beta\Delta v-v$.  Let $N(A_\beta)$ and $R(A_\beta)$ denote the kernel and the range of $A_\beta$ respectively. Note that $v\in N(A_\beta)$ if and only if $\Delta^2v-\beta\Delta v=v$ in $\Omega.$  Let $\varphi_1$ be the 
first eigenfunction of the Laplacian in $\Omega$, see \eqref{first:eigenfunction} 
with $\|\varphi_1\|_{L^2(\Omega)} = 1$. 

By the definition of $\bar\beta>0$, one has 
$\varphi_1\in N(A_{\bar\beta}).$ Moreover, by the Krein-Rutman Theorem 
 $N(A_{\bar\beta})=\{\alpha \varphi_1\::\: \alpha\in\R\}.$ Further, since $A_{\bar{\beta}}$ is  self adjoint, by the Fredholm Theory 
$R(A_{\bar\beta})=\{v\in Y : \int_\Omega \varphi_1 v \, dx =0\}.$  In particular, $\frac{d}{d\beta} A_{\beta}\varphi_1\mid_{\beta=\bar\beta}  = -\Delta\varphi_1= \lambda_1\varphi_1\not\in R(A_{\bar\gamma}).$ Hence, by \cite[Lemma 1.1]{crandall:1973} there are $\varepsilon>0$ and $C^1-$functions $\beta:(-\varepsilon,\varepsilon)\to(0,\infty)$ and $u:(-\varepsilon,\varepsilon)\to X$ such that $\beta(0)=\bar\beta$ and $G(\beta(t),u(t))=0$ for all $t\in(-\varepsilon,\varepsilon)$, moreover $G^{-1}(\{0\})$ near $(\bar\beta, 0)$ consists precisely of the curves $u\equiv0$ and $(\beta(t),u(t))$, $t\in(-\varepsilon,\varepsilon)$. Since $\partial_{uu}G(\bar\beta,0)[\varphi_1,\varphi_1]=0$ and $\int_\Omega \varphi_1 \partial_{uuu}G(\bar\beta,0)[\varphi_1,\varphi_1,\varphi_1]\, dx = 6\int_\Omega \varphi_1^4 \, dx>0$ we have a subcritical bifurcation, and therefore $u = \pm c (\bar{\beta} - \beta)^{\frac{1}{2}}\varphi_1+o( t^\frac{1}{2})$ for some constant $c>0$ and $t\in(0,\varepsilon)$. This proves the first claim.  

For the second claim, assume $\bar\beta>\sqrt{8}$ and that $\Omega$ is a smooth domain.  Let $(0,T)$ be the maximal interval of existence in $[\sqrt{8},\bar\beta]$ for the curve $u$ 
with $T\in(0,\infty]$. By this we mean that $\beta(t)\geq \sqrt{8}$ for $t\in(0,T)$, that $\gamma(t)$ can be uniquely extended for each $t\in(0,T)$, and, if $\beta(T)>\sqrt{8}$, then $\gamma$ cannot be uniquely extended at $T$. In particular, the curve ceases to exist if it intersects another curve e.g. $(\beta, 0)$ or if it bifurcates. 

We show first that $u(t)>0$ for all $t\in(0,T)$.  By the $C^1-$ continuity of the curve ${\cal C}:=\{u(t): t\in(0,T)\}$ we have that $0<u_\beta< 1$ for all $t\in(0,T)$ sufficiently close to zero. 
By Lemma \ref{u:bound:lemma} iii) and the continuity of ${\cal C}$ it follows that $\|u(t)\|_{L^\infty(\Omega)}\leq 1$ for all $t\in(0,T).$  
To show that $u(t)>0$ for all $t\in(0,T)$ we argue by contradiction. Assume that $\bar t=\sup\{t\in(0, T) : u(s)>0 \text{ in }\Omega \text{ for all }s\in(0,t]\}< T$.  Let $\bar u = u(\bar t)\in C^{4,\alpha}(\Omega).$  Note that $\bar u$ satisfies the system
\begin{equation*}
\begin{aligned}
 -\Delta \bar u + \beta \bar u = w \quad \text{ in }\Omega,\qquad -\Delta w = \bar u-\bar u^3 \quad \text{ in }\Omega,\qquad  w = u= 0\quad \text{ on }\partial\Omega,
\end{aligned}
\end{equation*}
for some $\beta>0,$ and $\bar u -\bar u^3\geq 0$ since $0\leq \bar u\leq 1.$ Since $T > \bar{t}$ we have that $\bar u \not\equiv 0$. Then, by the maximum principle and the Hopf Lemma, $\bar u >0$ in $\Omega$ and $\partial_\nu \bar u <0$ on $\partial \Omega$, where $\nu$ denotes the exterior normal vector to $\partial \Omega$. But this contradicts the $C^4-$continuity of $\cal C$ and the definition of $\bar u$.  Therefore $u(t)>0$ for all $t\in(0,T).$

Then, by Lemma \ref{u:bound:lemma}, $\|u(t)\|_{L^\infty(\Omega)}\leq 1$ for all $t\in(0,T)$, and standard elliptic regularity theory yields that $\|u(t)\|_{C^{4,\alpha}(\Omega)}\leq C$ for all $t\in(0,T)$ and for some $C>0$.  Moreover, since the positivity is preserved along the curve, $u(T) \not \equiv 0$. Indeed, $(\beta(t), u(t))$ cannot return to a neighborhood of $(\bar{\beta}, 0)$ by uniqueness close to $(\bar\beta,0)$. Also, it cannot intersect $(\beta, 0)$ as any other branch bifurcating from $(\beta, 0)$ $(\beta < \bar{\beta})$ consists locally of sign changing solutions because the corresponding eigenfunction directions are sign changing (perpendicular to the principal eigenfunction).  By the first part of Theorem \ref{existence:positive:solution}, we know that $\beta(t)<\bar\beta$ for all $t\in(0,T).$ This implies that necessarily $(\sqrt{8},\bar\beta)\subset \{\beta(t)\::\: t\in(0,T)\}$. This proves the existence of solutions for all $\beta\in[\sqrt{8},\bar\beta)$. 

We now show that $u_\beta$ is the unique positive solution of \eqref{efk:eq} for $\beta\in[\sqrt{8},\bar\beta)$. Indeed, let $v$ denote a positive solution of \eqref{efk:eq} for some  $\beta_0 \in[\sqrt{8},\bar\beta)$. By Theorem \ref{stability:thm}, $v$ is a strictly stable solution, and therefore $D_uG (v, \beta_0)$ is an invertible operator. Then, by the implicit function theorem there exists $\varepsilon > 0$  and a smooth curve $\gamma : (\beta_0 - \varepsilon, \beta_0 + \varepsilon) \to X$ such that $G(\beta, \gamma(\beta)) = 0$ and for any solution of $G(\beta, u) = 0$ sufficiently close to $(\beta_0, v)$ one has $u \in \gamma$.  

Arguing as before, we can extend $\gamma$ to a maximal interval $(\beta_1, \beta_2)$ with $\gamma$ containing only positive solutions. Then the strict stability in Theorem \ref{stability:thm} guarantees that $\gamma$ does not have bifurcation or turning points. Since the only solution for $\beta\geq \bar\beta$ is zero, by the first part of Theorem \ref{existence:positive:solution}, we have that $\beta_2 \leq \bar{\beta}$  and $\gamma (\beta_2)$ is a non-negative function. Arguing as before, one obtains that necessarily $\gamma(\beta_2)\equiv 0$ and then $\beta_2 = \bar{\beta}$. 
Here, as above,  we have used that all other bifurcation points of the form $(\beta,0)$ must correspond locally to sign changing solutions.  The uniqueness of the branch close to the bifurcation point $(\bar\beta,0)$ yields that necessarily $v = u_{\beta_0}$, as desired. 

If $\Omega$ is a hyperrectangle, one proves the positivity along the curve using Serrin's boundary point Lemma \cite[Lemma 1]{serrin:1971} at corners
and the rest of the proof remains unchanged.

\end{proof}

\begin{remark}\label{bif:pts}
For balls of radius $R>0$ we can explicitly write the relationship between $R$ and the bifurcation point $\bar\beta$. Indeed, in this case, $R:= \frac{\sqrt{2}j_{N/2-1,1}}{\sqrt{-\beta+\sqrt{\beta^2+4}}}$, where $N$ is the dimension and $j_{N/2-1,1}$ is the first positive zero of the Bessel function $J_{N/2-1}$, see for instance \cite[Section 4.1]{kesavan}.  For example, for $\beta=\sqrt{8}$, the bifurcation occurs at balls of radius $R_N:= \frac{j_{N/2-1,1}}{\sqrt{\sqrt{3}-\sqrt{2}}}$, for instance, $R_2:\approx 4.26$, $R_3:\approx 5.57$, $R_4:\approx 6.79$, $R_{10}:\approx 13.46$, etc...
\end{remark}

\section{Continuity result}\label{convergence:section}

\begin{proof}[Theorem \ref{convergence:theorem}]
Fix $p>N$, $\beta=1$, let $\gamma_0(\beta,p,\Omega)=\gamma_0>0$ be the constant given by Lemma \ref{l:reg:2}, let $0<\gamma<\min\{\gamma_0,\frac{1}{64}\}$, $u_\gamma$ be the global minimizer of \eqref{J:gamma},
 and $\mu=\gamma^{-\frac{1}{4}}.$ Note that $w:\mu\Omega\to \mathbb R$ given by $w(x):=u_\gamma(\mu^{-1}x)$ is a weak solution in $H$ of $\Delta^2 w - \mu^2\Delta w = w-w^3$ in $\mu\Omega.$ Also note that $\mu\geq \sqrt{8}$ if $\gamma\leq \frac{1}{64}.$  By Proposition \ref{bound:global:minimizers} and Lemma \ref{l:reg:2} we have that $\| u_\gamma \|_{L^\infty(\Omega)}=\|v\|_{L^\infty(\Omega)}\leq 1$ and $\|u_\gamma \|_{C^{5,\alpha}(\Omega)}\leq C$ for some $C>0$ independent of $\gamma.$

 Let $u^*\in H^1_0(\Omega)$ be a global minimizer of \eqref{J:gamma} in $H^1_0(\Omega)$ with $\gamma=0$. It is well known that $u^*$ is a unique global minimizer, smooth, strictly stable, and it does not change sign (see, for example, \cite{berestycki:1981}).

Now, since $u_\gamma$ is bounded in $C^{5,\alpha}$ independently of $\gamma$ it is easy to see that $u_\gamma \to u^*$ in $C^4$ as $\gamma\to 0$, by the
uniqueness of global minimizers of the limit problem \eqref{efk:gamma} with $\gamma=0$. 

Let $G\in C^1(\R \times C^{4}(\Omega))$ be given by $G(\gamma,u)=\gamma\Delta^2 u - \Delta u -u + u^3.$ Notice that $\partial_u G (\gamma,u) \in {\cal L}(C^{4}(\Omega),\R)$ and $\partial_u G (0,u^*)$ has trivial kernel by the strict stability of $u^*$. Therefore, by the implicit function theorem (see for example \cite[Theorem I.1.1]{kielhofer:2012}), there is a neighborhood $I\times V\subset \R \times C^{4}(\Omega)$ of $(0,u^*)$ and a continuous function $\lambda:I\to V$ with $\lambda(0)=u^*$ such that $G(\gamma,\lambda(\gamma))=0$ for all $\gamma\in I$ and every solution of $G(\gamma,u)=0$ in $I\times V$ is of the form $(\gamma,\lambda(\gamma))$ for some $\gamma\in I.$  Since $u_\gamma\to u^*$ in $C^{4}$ as $\gamma\to 0^+$ and $u_\gamma$ is an arbitrary global minimizer, we obtain that $u_\gamma$ is the unique global minimizer for all $\gamma\in I$. Finally, if the first Dirichlet eigenvalue $\lambda_1(\Omega)<1,$ then $u^*\not\equiv 0$ (cf. proof of Theorem \ref{existence:positive:solution}) and by the strong maximum principle $u^*>0$ in $\Omega,$ by the Hopf Lemma $\partial_\nu u^*<0$ in $\partial \Omega$, where $\nu$ denotes the exterior normal vector, and therefore $u_\gamma>0$ in $\Omega$ for all $\gamma\in I$, by making $I$ a smaller neighborhood of $0$ if necessary.

\end{proof}

\begin{remark}
 Note that the proof of Theorem \ref{convergence:theorem} also shows the existence of solutions in $C^{5,\alpha}(\Omega)$ for equation \eqref{efk:gamma} with $\gamma\in[-\gamma_0,0]$. 
\end{remark}

\section{Numerical approximations}\label{numerics}
In this Section we present some numerical approximations of a bifurcation branch with respect to $\beta$ from $\bar\beta=\frac{1-\lambda_1^2}{\lambda_1}$ for the problem \eqref{efk:eq} (cf. Theorem \ref{bifurcation:beta}). These approximations were computed with the software AUTO-07P \cite{auto}.

\begin{figure}[h!]
   \centering
 \includegraphics[width=11cm,height=7cm]{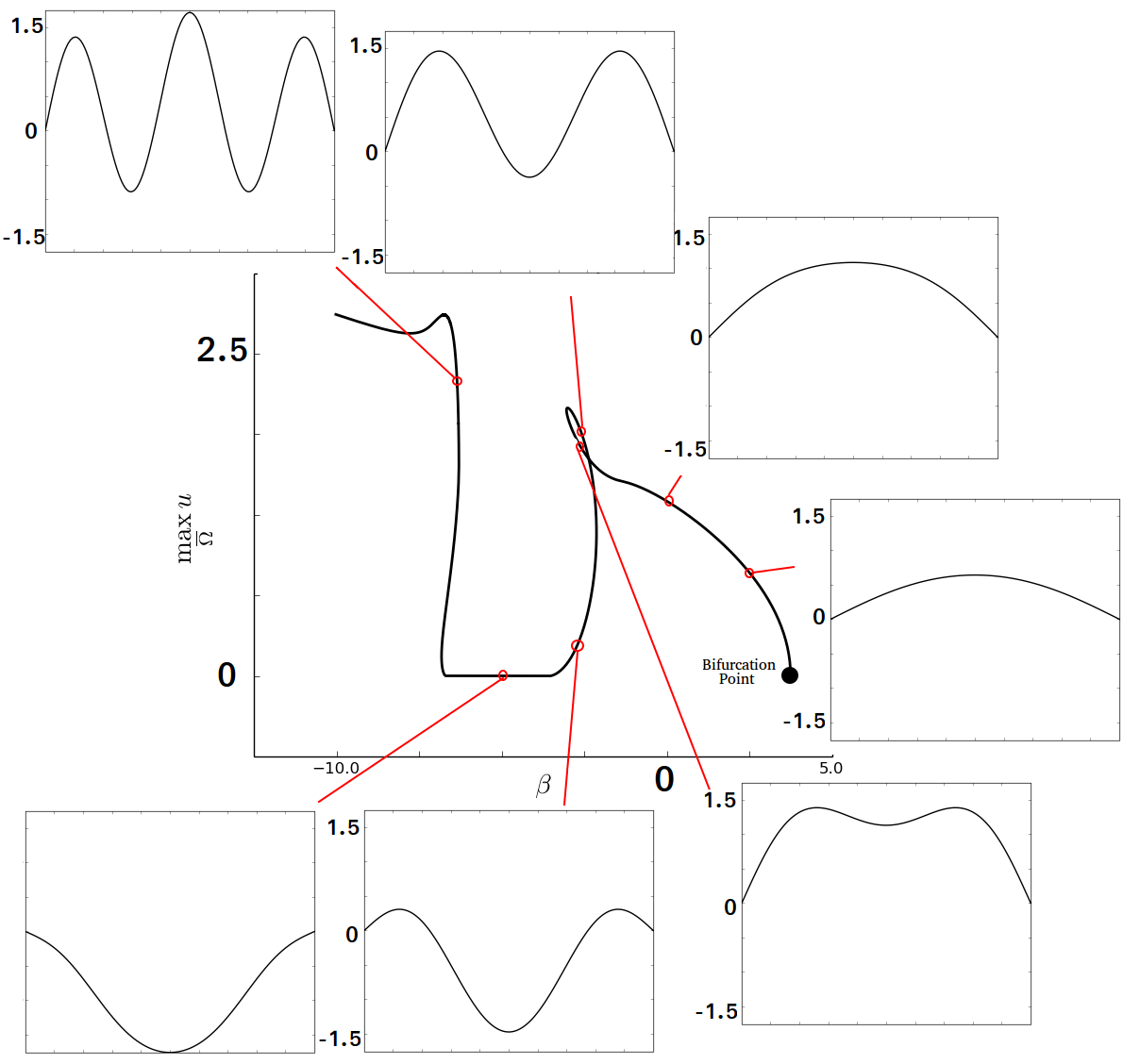}
   \caption{\footnotesize{Here $\Omega=(0,2\pi)$. In this picture we present approximations of solutions of \eqref{efk:eq} along the branch. Note that for $\beta<0$ the positivity is lost and even negative solutions appear. Here the branch does not return to the region $\beta>0$. 
   }}
\end{figure}

\begin{figure}[h!]
   \centering
 \includegraphics[height=4cm]{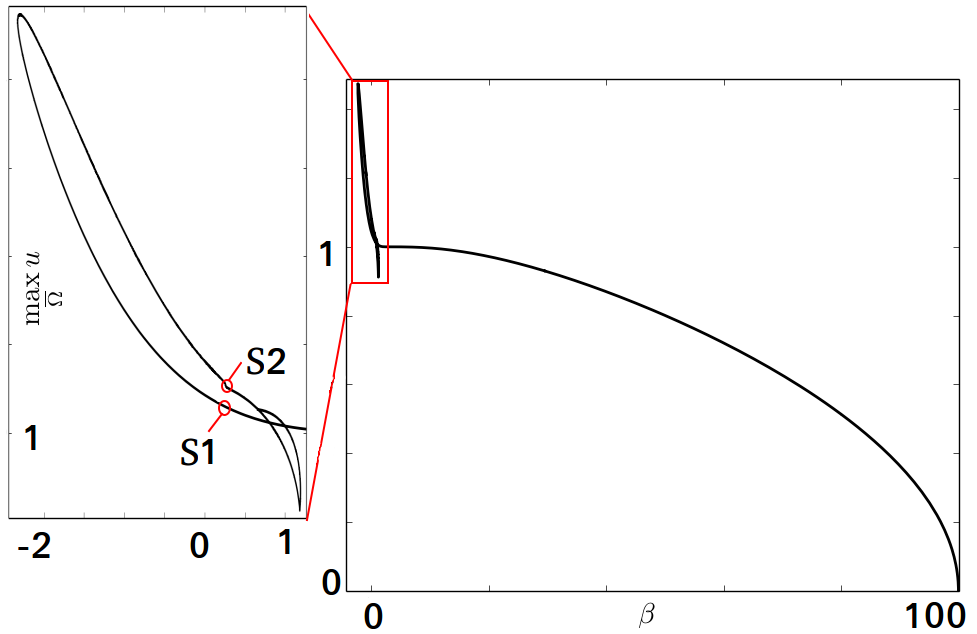} \includegraphics[height=4.9cm]{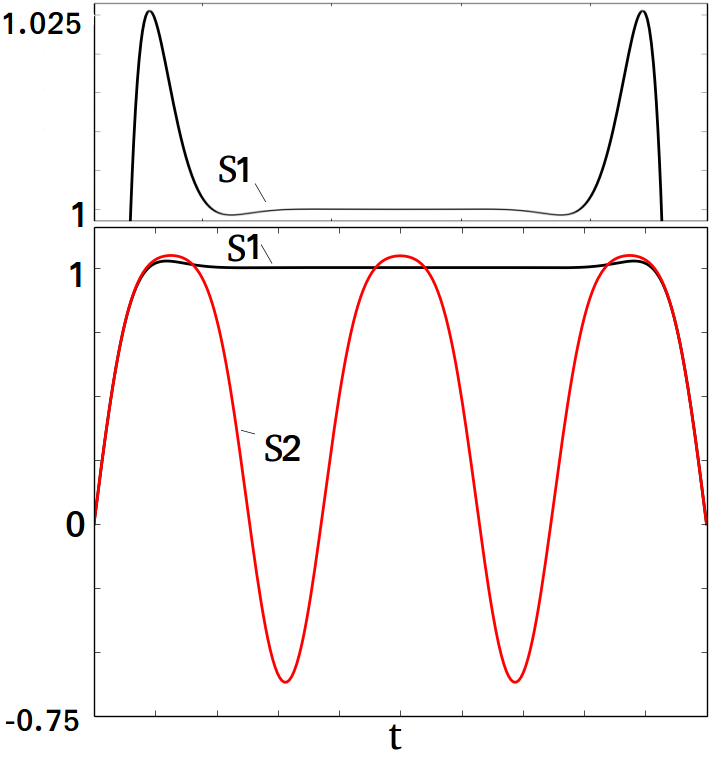}
   \caption{\footnotesize{Here $\Omega = (0, 10\pi)$. In this picture we show numerical evidence that positive solutions are not bounded by one and present oscillations. Also that the bifurcation branch may return to positive values of $\beta$ (although the positivity is lost). On the right we have an approximation of two solutions of \eqref{efk:eq} along the branch corresponding to $\beta=0.35$.}}
 \end{figure}

\newpage

\def\cprime{$'$}

\end{document}